\documentclass[12pt]{article}
\usepackage[utf8]{inputenc}
\usepackage[a4paper, left=1in,right=1in]{geometry}

\usepackage[numbers]{natbib}
\RequirePackage{amsthm, amsmath, amsfonts, amssymb}
\RequirePackage{graphicx, setspace, authblk}
\RequirePackage[colorlinks,citecolor=blue,urlcolor=blue]{hyperref}
\theoremstyle{plain}
\newtheorem{theorem}{Theorem}
\newtheorem{proposition}[theorem]{Proposition}
\newtheorem{lemma}[theorem]{Lemma}

\theoremstyle{remark}

\newtheorem{definition}[theorem]{Definition}
\newtheorem{example}[theorem]{Example}
\newtheorem{assumption}[theorem]{Assumption}

\newcommand{\norm}[1]{\left\lVert#1\right\rVert}

\newcommand{\e}{\epsilon}

\newcommand{\tr}{\text{tr}}

\newcommand{\R}{\mathbb{R}}
\newcommand{\Z}{\mathbb{Z}}
\newcommand{\Q}{\mathcal{Q}}

\newcommand{\E}{\mathbb{E}}
\newcommand{\Prob}{\mathbb{P}}
\renewcommand{\P}{\mathcal{P}}

\newcommand{\W}{\mathcal{W}}
\newcommand{\A}{\mathcal{A}}
\newcommand{\K}{\mathcal{K}}

\renewcommand{\H}{\mathcal{H}}

\newcommand{\X}{\mathcal{X}}
\newcommand{\Y}{\mathcal{Y}}
\newcommand{\B}{\mathcal{B}}

\newcommand{\C}{\mathcal{C}}
\renewcommand{\S}{\mathcal{S}}

\let\temp\phi
\let\phi\varphi
\let\varphi\temp

\newcommand{\eq}[1]{\begin{align*}#1\end{align*}}

\begin{document}

\title{Upper and lower bounds on the subgeometric convergence of adaptive Markov chain Monte Carlo}
%\runtitle{Bounds on the subgeometric convergence of adaptive MCMC}

\author[1]{Austin Brown \thanks{email: austinbrown@tamu.edu}}
\author[2]{Jeffrey S. Rosenthal \thanks{email: jeff@math.toronto.edu}}
\affil[1]{Department of Statistics, Texas A\&M University, College Station, Texas, USA}
\affil[2]{Department of Statistical Sciences, University of Toronto, Toronto, Canada}

\maketitle

\begin{abstract}
We investigate lower bounds on the subgeometric convergence of adaptive Markov chain Monte Carlo under any adaptation strategy. In particular, we prove general lower bounds in total variation and on the weak convergence rate under general adaptation plans. 
If the adaptation diminishes sufficiently fast, we also develop comparable convergence rate upper bounds that are capable of approximately matching the convergence rate in the weak subgeometric lower bound.
These results provide insight into the optimal design of adaptation strategies and also limitations on the convergence behavior of adaptive Markov chain Monte Carlo. 
Applications to an adaptive unadjusted Langevin algorithm as well as adaptive Metropolis-Hastings with independent proposals and random-walk proposals are explored. 
\end{abstract}



\section{Introduction}

Let $\pi$ be a Borel probability measure on a Polish space $\X$.
Adaptive Markov chain Monte Carlo \citep{haario:2001, roberts:rosenthal:2007} is a widely successful framework to simulate realizations from $\pi$ when optimal tuning parameters for the Markov chain are not readily available.
The adaptive process $(\Gamma_t, X_t)_{t = 1}^\infty$ is constructed from a family of Markov kernels indexed by a set of potential tuning parameters.
The discrete-time adaptive process first updates the tuning parameter $\Gamma_t | ( \Gamma_s, X_s )_{0 \le s \le t-1}$ with an adaptation strategy utilizing previous history and next, updates $X_t | \Gamma_{t}, X_{t - 1}$ using a Markov transition kernel.
The goal is for the adaptive process to ``learn" optimal tuning parameters so that the marginal distribution of the random variable $X_t$ produces a close approximation to the measure $\pi$.

With a large option for adaptation strategies, theoretical convergence rates of adaptive algorithms are less understood than for non-adaptive Markov chain Monte Carlo (MCMC) where fixed tuning parameters are chosen carefully beforehand.
In particular, a theoretical understanding of the rate of convergence is essential in applications as it helps to ensure a stable and reliable Monte Carlo simulation.
However, adaptive MCMC can exhibit empirical performance superseding the performance of standard MCMC even though much of the theoretical understanding is lacking.
For example, adaptive MCMC is widely used to automatically learn the covariance in random-walk Metropolis-Hastings \citep{haario:2001}, which is often difficult or impossible to choose optimally with only fixed tuning parameter choices. 

The main contributions of this paper develop general subgeometric lower bounds in total variation and the weak convergence rate of adaptive MCMC paired with upper bounds under strong conditions on the rate at which adaptation diminishes.
Applications of the theory are demonstrated on an adaptive unadjusted Langevin algorithm, Metropolis-Hastings independence sampler, and an adaptive Metropolis-Hastings random-walk.
The lower bounds for convergence hold under arbitrary adaptation plans and serve as a measurement of the optimal convergence behavior for adaptive MCMC.
The techniques for obtaining these lower bounds are based on finding large discrepancies between the tail probabilities of the marginal adaptive process and the target measure $\pi$.
Since the convergence rate is determined by tail properties, this may guide further theoretical understanding of some modern adaptation strategies that restrict adaptation to compact sets \citep{pompe:etal:2020}.
Convergence rate lower bounds can also be of practical use in applications to determine if an appropriate rate is achievable so that central limit theorems may hold \citep{andrieu:moulines:2006, laitinen:vihola:2024}.

One barrier in developing lower bounds for adaptive MCMC is due to the non-Markovian, non-reversible nature of these processes and spectral analysis for reversible Markov processes is not directly available.
To the best of our knowledge, the lower bounds for weak convergence developed here are novel, even when applied to non-adapted Markov chains, and general total variation lower bounds have not yet been explored for adaptive MCMC.
In specific situations, adaptive random-walk algorithms have been shown to improve ``local" behavior but fail to adapt to ``global" properties of the target measure, such as the tail probabilities, and proven to experience poor convergence properties \citep{schmidler:woodard:2011}.
Related research develops general lower bounds in total variation for Markov processes \citep[Theorem 3.6, Corollary 3.7]{hairer:2009}.
More recently, this technique has also been extended to polynomial rate lower bounds in unbounded Wasserstein distances for some Markov processes \citep[Theorem 1.2]{sandric:etal:2022}.
When the tail decay of the target measure is unavailable, lower bounds for Markov processes in total variation have recently been developed, but a precise computation of the constants is not available \citep{bresar:etal:2024}.
%Upper bounds on the convergence rate of adaptive MCMC have been explored under restrictive conditions \citep{andrieu:moulines:2006} and more generally, martingale techniques have been used \citep{laitinen:vihola:2024}.

In addition to lower bounds, we develop explicit quantitative subgeometric upper bounds in total variation that can match the lower bound rate if the adaptation diminishes sufficiently fast.
The condition required on the adaptation is similar to the well-known diminishing adaptation condition \citep{roberts:rosenthal:2007} often used for the asymptotic convergence of adaptive MCMC.
%This provides a notion of an optimal adaptation strategy, which is useful in the future design of adaptive MCMC simulations.
To the best of our knowledge, this is the first subgeometric upper bound to quantify the mixing for adaptive MCMC in total variation. 
In comparison, existing convergence results are asymptotic \citep{yves:fort:2010}, require strong assumptions for adaptive MCMC \citep{andrieu:atach2007, andrieu:moulines:2006}, or develop central limit theorems through Poisson's equation \citep{laitinen:vihola:2024}.
%However, applications of the explicit upper bounds are limited due to the requirement of a finite adaptation plan.

The organization of this article is as follows.
Some preliminary definitions are presented in Section~\ref{section:prelim} and Section~\ref{section:lower_bounds} first develops lower bounds in total variation for large classes of adaptation strategies and then extends these lower bounds to weak convergence when the state space is Euclidean.
A lower bound is shown on a concrete example for the adapted unadjusted Langevin algorithm.
Section~\ref{section:upper_bounds} proves comparable upper bounds under diminishing conditions on the adaptation plans that are capable of approximately matching the lower bound rates.
Section~\ref{section:toy_example} illustrates the lower bounds on a toy example with an adaptive Metropolis-Hastings independence sampler, and Section~\ref{section:adaptive_rwm} applies the lower bounds to the popular adaptive random-walk Metropolis-Hastings.
Section~\ref{section:conclusion} provides a final discussion on the results and future research directions.

\section{Preliminaries}
\label{section:prelim}

For two Borel probability measures $\mu, \nu$ on $\X$, let $\C(\mu, \nu)$ be the set of all couplings consisting of Borel probability measures on $\X \times \X$ satisfying $\Gamma(\cdot \times \X) = \mu$ and $\Gamma(\X \times \cdot) = \nu$. Denote then the total variation distance between $\mu$ and $\nu$ as the best probability of the off-diagonal over all possible couplings, that is,
\[
\norm{ \mu - \nu }_{\text{TV}}
= \inf_{\xi \in \C(\mu, \nu)} \xi( \{ (x, y) \in \X \times \X : x \not= y \} ).
\]
Denote the min and max of $a, b \in \R$ by $a \wedge b$ and $a \vee b$ respectively.
On a Polish space $(\X, d)$ where $d : \X \times \X \to [0, \infty)$ is a metric, we denote the Wasserstein distance that metrizes the weak convergence of probability measures \citep[Theorem 11.3.3]{dudley:2018}
\[
\W_{d \wedge 1}(\mu, \nu)
= \inf_{\xi \in \C(\mu, \nu)} \int_{\X \times \X} \left[ d(x, y) \wedge 1 \right] \xi(dx, dy).
\]

Let $\X$ be a Polish space and $\Y$ be a Borel measurable space equipped with their Borel sigma-algebras $\B(\X)$ and $\B(\Y)$ respectively where $\X$ is the state space and $\Y$ is the space for tuning parameters.
We now define the adaptive process $(\Gamma_t, X_t)_{t = 0}^\infty$ on $\Y \times \X$ using the filtration $\H_t = \B( \Gamma_s, X_s, 0 \le s \le t )$.
Let $\Q$ define an adaptation plan which denotes the map $t \mapsto \Q_t$ for all $t \in \Z_+$ where $\Q_t : (\Y \times \X)^{t} \times \B(\Y) \to [0, 1]$ is a Borel probability kernel.
The kernels $\Q_t$ act on Borel functions $g : \Y \to \R$ and Borel measures $\nu$ on $(\Y \times \X)^{t}$ with 
\eq{
&(\Q_t g)(\gamma_0, x_0, \ldots, \gamma_{t-1}, x_{t-1}) = \int_{\X} g(\gamma_t) \Q_t( \gamma_0, x_0, \ldots, \gamma_{t-1}, x_{t-1}, d\gamma_t )
\\
&(\nu \Q_t)(\cdot) = \int_{\X} \Q_t(\gamma_0, x_0, \ldots, \gamma_{t-1}, x_{t-1}, \cdot) \nu(d\gamma_0, dx_0, \ldots, d\gamma_{t-1}, dx_{t-1})
}
for all $t \in \Z_+$ and $\gamma_0, x_0, \ldots, \gamma_{t-1}, x_{t-1} \in (\Y \times \X)^t$.
Initialized at fixed $x_0, \gamma_0 \in \X \times \Y$, the discrete-time adaptive process first updates the tuning parameter 
\[
\Gamma_t | ( \Gamma_s, X_s )_{0 \le s \le t-1} \sim \Q_t(( \Gamma_s, X_s )_{0 \le s \le t-1}, \cdot)
\]
using an adaptation plan.
Let $(\P_\gamma)_{\gamma \in \Y}$ be a family of Borel Markov kernels where $\P_\gamma : \X \times \B(\X) \to [0, 1]$ for each $\gamma \in \Y$ and for each $x \in \X$, $x, \gamma \mapsto \P_\gamma(x, \cdot)$ is Borel measurable.
The Markov family acts on Borel functions $f : \X \to \R$ and Borel measures $\mu$ on $\X$ with 
\eq{
&(\P_\gamma f)(x) = \int_{\X} f(y) \P_\gamma(x, dy)
&(\mu \P_\gamma)(\cdot) = \int_{\X} \P_\gamma(x, \cdot) \mu(dx)
}
for all $(x, \gamma) \in \X \times \Y$.
The process then updates the state space given the updated tuning parameters 
\[
X_t | \Gamma_{t}, X_{t - 1} \sim \P_{\Gamma_{t}}(X_{t-1}, \cdot)
\]
for $t \ge 1$ using the Markov kernel.

\section{Lower bounds on the convergence of adaptive MCMC}
\label{section:lower_bounds}

Let $\S(\X, \Y)$ denote the set of all possible adaptation plans $\Q$ that define the Borel kernels $\Q_t$ updating the tuning parameters at every iteration time $t$.
For a chosen adaptive strategy $\Q \in \S(\X, \Y)$, we denote the marginal of the adaptive process at iteration time $t$ by $X_t \sim \A_\Q^{(t)}((\gamma_0, x_0), \cdot)$.
We will develop conditions to lower bound the total variation over all feasible adaptation strategies, that is, to lower bound
\[
\inf_{\Q \in \S(\X, \Y) } \norm{ \A_\Q^{(t)}(\gamma_0, x_0, \cdot) - \pi }_{\text{TV}}
\]
for $t \in \Z_+$.

The main tool will be a function prescribing a subgeometric rate defined implicitly as an inverse which we now define.
For concave functions $\phi : (0, \infty) \to (0, \infty)$ and $w_0 \in [1, \infty)$, define
\begin{align}
H_{w_0, \phi}(w) = \int_{w_0}^w \frac{dv}{\phi(v)}
\end{align}
for all $w \ge w_0$.
The assumptions on $\phi$ imply it is non-decreasing and $H_{w_0, \phi}(\cdot)$ is strictly increasing as well as the inverse $H^{-1}_{w_0, \phi}(\cdot)$ exists.
Depending on the form of $\phi$, the inverse function $H^{-1}_{w_0, \phi}(\cdot)$ defines a polynomial, subgeometric, or geometric function increasing to infinity.

The first lower bound in total variation uses a technique extended from \citep[Corollary 3.7]{hairer:2009} to adaptive MCMC over all adaptive strategies.

\begin{theorem}
\label{theorem:lb_with_target_tails}
Assume there is a Borel function $W : \X \to [1, \infty)$ and constants $C \in (0, 1]$ and $\kappa \in (0, \infty)$ where
\begin{align}
\pi( W \ge r ) \ge C r^{-\kappa}
\label{eq:target_tail_lb}
\end{align}
holds for all $r \ge 1$ and there is an $\alpha > \kappa$ and a concave function $\phi : (0, \infty) \to (0, \infty)$ such that
\begin{align}
&(\P_\gamma W^{\alpha} )(x) - W(x)^\alpha  \le \phi( W(x)^\alpha )
\label{eq:lb_drift}
\end{align}
holds for all $(x, \gamma) \in \X \times \Y$.
Then for all $t$,
\eq{
\inf_{\Q \in \S(\X, \Y) } \norm{ \A_\Q^{(t)}(\gamma_0, x_0, \cdot) - \pi }_{\text{TV}}
\ge \frac{M}{ \left( H_{W(x_0)^\alpha, \phi}^{-1}(t) \right)^{\frac{\kappa}{\alpha - \kappa}}}
}
where 
\begin{align}
M = C^{\frac{\alpha}{\alpha - \kappa}} \left[ ( \kappa/\alpha )^{\frac{\kappa}{\alpha - \kappa}} - ( \kappa/\alpha )^{\frac{\alpha}{\alpha - \kappa}} \right].
\label{eq:M}
\end{align}
\end{theorem}

\begin{proof}
Let $V(x) = W^\alpha(x)$, and let $t \in \Z_+$, so then we have
\[
\E\left( V(X_{t + 1}) | \H_t \right) - V(X_{t})
\le \phi(V(X_{t})).
\]
Since $\E\left[ V(X_{1}) \right] - V(x_0) \le \phi( V(x_0) )$, then assume by induction for all $1 \le k \le t-1$,
$
\E\left[ V(X_{k + 1}) \right] - \E[ V(X_{k}) ] \le \phi(\E[ V(X_{k}) ] )
$
and $\E[ V(X_{k}) ] < \infty$.
By the induction hypothesis and Jensen's inequality,
\begin{align}
\E\left[ V(X_{t + 1}) \right] - \E\left[ V(X_{t}) \right]
&= \E\left[ 
\E\left( V(X_{t + 1}) | \H_t \right) - V(X_{t})
\right]
\nonumber
\\
&\le \E\left(
\phi[ V(X_{t}) ]
\right)
\nonumber
\\
&\le
\phi( \E\left[ V(X_{t} ) \right] ).
\label{eq:lb_lyap_ineq}
\end{align}

The inverse function theorem implies the derivative 
\[
\frac{d}{ds} H^{-1}_{V(x_0), \phi}(s) = \phi(H^{-1}_{V(x_0), \phi}(s)).
\]
Since $H^{-1}_{V(x_0), \phi}(0) \ge V(x_0)$, assume by induction $H^{-1}_{V(x_0), \phi}(k) \ge \E[ V(X_{k}) ]$ for all $k \le t$.
Since $\phi$ is non-decreasing combined with the fundamental theorem of calculus and \eqref{eq:lb_lyap_ineq},
\begin{align*}
H^{-1}_{V(x_0), \phi}( t + 1 )
= H^{-1}_{V(x_0), \phi}( t ) + \int_{t}^{t + 1}  \phi( H^{-1}_{V(x_0), \phi}( s ) ) ds
&\ge H^{-1}_{V(x_0), \phi}( t ) + \phi( H^{-1}_{V(x_0), \phi}( t ) )
\\
&\ge \E\left[ V(X_{t}) \right] + \phi( \E\left[ V(X_{t}) \right] )
\\
&\ge \E\left[ V(X_{t + 1}) \right].
\end{align*}
By Markov's inequality,
\[
\Prob( W(X_t) \ge r)
\le \frac{ \E\left[ W(X_t)^\alpha \right] }{ r^{\alpha} }
\le \frac{ H_{W(x_0)^\alpha, \phi}^{-1}(t) }{ r^{\alpha} }.
\]
We choose $r^{\alpha - \kappa} = (\alpha / \kappa) \cdot H_{W(x_0)^\alpha, \phi}^{-1}(t) / C$ so that $r \ge 1$ and gives the lower bound
\eq{
\norm{ \A_{\Q}^{(t)}(\gamma_0, x_0, \cdot) - \pi }_{\text{TV}}
&\ge \pi(W \ge r) - \Prob(W(X_t) \ge r)
\ge \frac{C}{r^\kappa} - \frac{ H_{W(x_0)^\alpha, \phi}^{-1}(t) }{ r^{\alpha} }
\\
&\ge \frac{M}{ \left( H_{W(x_0)^\alpha, \phi}^{-1}(t) \right)^{\frac{\kappa}{\alpha - \kappa} } }.
}
\end{proof}

Assumption \eqref{eq:lb_drift} of Theorem~\ref{theorem:lb_with_target_tails} requires the Markov family $(\P_\gamma)_{\gamma \in \Y}$ to satisfy a simultaneous growth condition for some concave function $\phi$.
 %This assumption is analogous to simultaneous subgeometric drift conditions used to show asymptotic convergence in adaptive MCMC and can provide new insights into many unknown convergence rates \citep{roberts:rosenthal:2007}. 
We look at some concrete examples of concave functions that lead to common subgeometric convergence rates that have been explored previously for upper bounds \citep{douc:etal:2004}.

\begin{example}
(Polynomial lower bounds)
Assume \eqref{eq:target_tail_lb} holds with constants $C > 0$ and $\kappa = 1$ and additionally, \eqref{eq:lb_drift} holds with function $W(\cdot)$, $\alpha = 2$, and $\phi(w) = c w^\beta$ for some constants $c > 0$ and $\beta \in (0, 1)$. Then a straight-forward calculation gives
$
H^{-1}_{W(x_0)^2, \phi}(t) = \left( (1 - \beta) c t + W(x_0)^{2( 1-\beta )} \right)^{\frac{1}{1 - \beta}}
$
and Theorem~\ref{theorem:lb_with_target_tails} implies for all $t \in \Z_+$,
\eq{
\inf_{\Q \in \S(\X, \Y) }
\norm{ \A_\Q^{(t)}(\gamma_0, x_0, \cdot) - \pi }_{\text{TV}}
\ge \frac{C^2}{ 4 \left( (1 - \beta) c t + W(x_0)^{2( 1-\beta )} \right)^{\frac{1}{1 - \beta}} }.
}
\end{example}

\begin{example}
(Subgeometric lower bounds)
If \eqref{eq:target_tail_lb} holds with constants $C > 0$, $\beta \in (0, 1)$, and $\kappa = 1$ and \eqref{eq:lb_drift} holds with $W(\cdot)$, $\alpha = 2$, and $\phi(x) = c (x + K_\beta)/\log(x + K_\beta )^\beta$ where $K_\beta = \exp(\beta + 1)$, then
%\[
%\int_{W(x_0)^2}^w 1/\phi(v) dv
%= \frac{ \log( w + K_\beta )^{1 + \beta} - \log( W(x_0)^2 + %K_\beta )^{1 + \beta} }{ c ( 1 + \beta )}
%\]
%and therefore, this defines
\[
H^{-1}_{W(x_0)^2, \phi}(t)
\le (W(x_0)^2 + K_\beta) \exp\left( [ (1 + \beta) c t ]^\frac{1}{1 + \beta} \right).
\]
By Theorem~\ref{theorem:lb_with_target_tails}, then for all $t \in \Z_+$
\eq{
\inf_{\Q \in \S(\X, \Y) }
\norm{ \A_\Q^{(t)}(\gamma_0, x_0, \cdot) - \pi }_{\text{TV}}
\ge \frac{C^2 }{ 4(W(x_0)^2 + K_\beta) } \exp\left( -[ (1 + \beta) c t ]^\frac{1}{1 + \beta} \right).
}
\end{example}

Now we obtain a matching weak lower bound rate as total variation under additional conditions.

\begin{theorem}
\label{theorem:weak_lb_with_target_tails}
Assume \eqref{eq:target_tail_lb} holds with $C, \kappa$ and \eqref{eq:lb_drift} holds with $W(\cdot)$, and $\alpha$, and let $M$ be defined as in \eqref{eq:M}.
Assume $\log(W)$ is Lipschitz continuous.
Then for any $\e \in (0, 1)$,
\begin{align*}
\inf_{\Q \in \S(\X, \Y)} \inf_{\xi \in \C\left[ \A_{\Q}^{(t)}(\gamma_0, x_0, \cdot), \pi \right] } \xi( \{ x, y \in \X \times \X : d(x, y) > \delta_{\e} \})
\ge \frac{(1 - \e) M}{ H_{W(x_0)^\alpha, \phi}^{-1}(t)^{\frac{\kappa}{\alpha - \kappa}} }
\end{align*}
holds for some $\delta_{\e} \in (0, 1)$.
\end{theorem}

\begin{proof}
Let $r \ge 1$ and let $T = \{x \in \X : W(x) \ge r \}$.
Since $\log(W)$ is continuous, then $T$ is closed and by Strassen's theorem (\citep{strassen:1965} and \citep[Corollary 1.28]{villani:2003}), then for any $\delta > 0$,
\[
\inf_{ \xi \in \C\left[ \A_{\Q}^{(t)}(\gamma_0, x_0, \cdot), \pi \right] } 
\xi( \{ x, y \in \X \times \X : d(x, y) > \delta \})
\ge \pi(T) - \A_{\Q}^{(t)}(\gamma_0, x_0, T^\delta)
\]
where $T^\delta = \{y \in \X : \text{dist}(y, T) \le \delta \}$ and $\text{dist}(y, T) = \inf_{x \in T} d(x, y)$.
Thus, we will find a discrepancy between $\pi\left( \{ W \ge r \} \right)$ and $\A_{\Q}^{(t)}\left( \gamma_0, x_0, \{ W \ge (1 - \e) r \} \right)$ for small $\e$ and the intuition is illustrated in Figure~\ref{figure:diagram}.
\begin{figure}
\centering
\includegraphics[width=.9\linewidth]{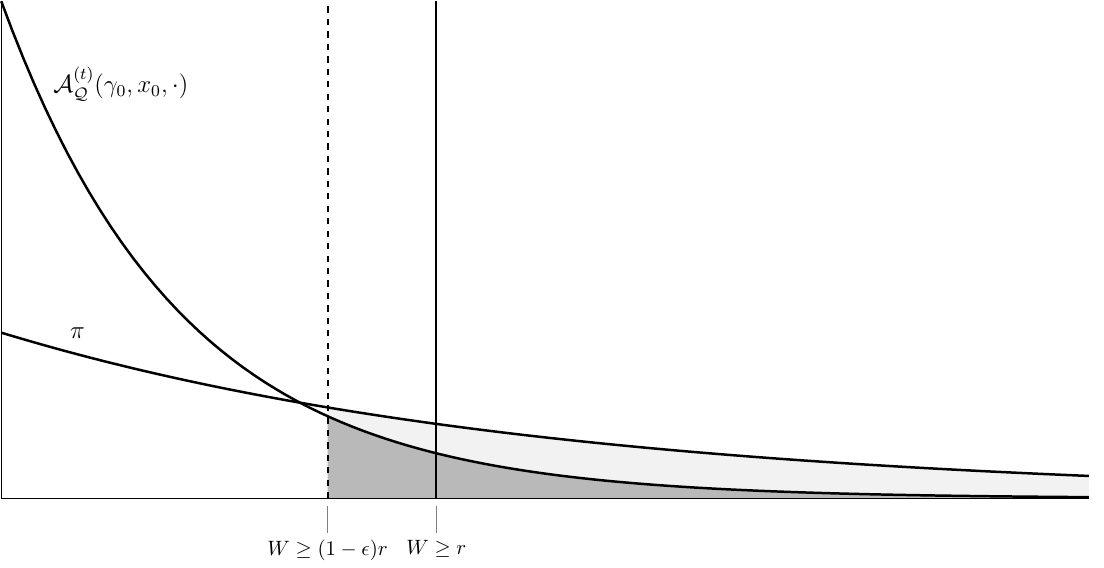}
\caption{
The diagram illustrates intuition for a discrepancy between the set $\{ W \ge r \}$ for the adaptive process and the target measure and also $\{ W \ge (1 - \e) r \}$  for small $\e$.}
\label{figure:diagram}
\end{figure}

For $\e \in (0, 1)$, since $\log(W)$ is Lipschitz continuous, then we can then choose $\delta_\e$ depending on $\e$ sufficiently small so that $W(x) \ge (1 - \e) r$ if $\text{dist}(x, T) \le \delta_\e$ and so
\begin{align*}
\Prob(X_t \in T^{\delta_\e}) 
&\le \Prob( W(X_t) \ge (1 - \e) r ).
\end{align*}
Markov's inequality and \eqref{eq:lb_drift} imply that
\[
\Prob( W(X_t) \ge (1 - \e) r )
\le \frac{\E\left[ W^\alpha(X_t) \right] }{ (1 - \e)^{\alpha}  r^{\alpha} }
\le \frac{ H_{W(x_0)^\alpha, \phi}^{-1}(t) }{ (1 - \e)^{\alpha} r^{\alpha} }.
\]
Optimizing, we get for $t$ large enough so that
\[
r = \left( \frac{\alpha}{\kappa C (1 - \e)^{\alpha} } H_{ W(x_0)^\alpha, \phi}^{-1}(t) \right)^{\frac{1}{\alpha - \kappa}}
\ge 1
\]
and this yields the lower bound
\begin{align*}
\inf_{\xi \in \C\left[ \A_{\Q}^{(t)}(\gamma_0, x_0, \cdot), \pi \right] } \xi( \{ x, y : d(x, y) > \delta_\e \})
&\ge \frac{C}{ r^\kappa } - \frac{H_{W(x_0)^\alpha, \phi}^{-1}(t) }{ (1 - \e)^{\alpha} r^{\alpha} }.
\\
&\ge (1 - \e)^{\frac{\alpha \kappa}{\alpha - \kappa}} \frac{M}{ H_{W(x_0)^\alpha, \phi}^{-1}(t)^{\frac{\kappa}{\alpha - \kappa}} }
\end{align*}
where $M$ is defined by \eqref{eq:M}.
The conclusion follows since we may choose $\e$ so that $(1 - \e)^{\frac{\alpha \kappa}{\alpha - \kappa}} \ge 1-a$ for any $a \in (0, 1)$ and some $\delta_{a}$.
\end{proof}

An interpretation of Theorem~\ref{theorem:weak_lb_with_target_tails} is the best possible rate of convergence for adaptive MCMC satisfying \eqref{eq:lb_drift} for target measure satisfying \eqref{eq:target_tail_lb}.
As can be seen in the proof, the conclusion of Theorem~\ref{theorem:weak_lb_with_target_tails} holds generally for functions $W$ that are not necessarily log-Lipschitz but satisfy for any $\e > 0$ and $r \ge 1$, if $d(x, \{ W \ge r \}) < \delta_\e$ for some $\delta_\e$, then $W \ge (1 - \e) r$.
The assumption on the function $W$ holds in many applications. However, there is a significant drawback to the Wasserstein lower bound being as the constant is non-explicit compared to the explicit lower bound in total variation.

What is surprising about the lower bounds in this section is the requirement only on the Markov family $(\P_\gamma)_{\gamma \in \Y}$ to satisfy \eqref{eq:lb_drift} and does not directly depend on an adaptation strategy.
For example, it is common scenario in adaptive MCMC for the parameter space $\Y$ to be compact.
In this case, the simultaneous growth condition \eqref{eq:lb_drift} often holds if a Markov kernel satisfies some mild regularity conditions and \eqref{eq:lb_drift} holds with only fixed parameters.

\begin{example} 
\label{section:ula}
(Adaptive Unadjusted Langevin algorithm)
Consider the multivariate Student's t-distribution $\pi$ on $\R^d$ with $d \ge 1$ and $v > 0$ degrees of freedom. The Lebesgue density is defined by
\[
D_{\pi}(x)
= \frac{\Gamma( (v + d)/2 )}{\Gamma(v/2) (v \pi)^{d/2}} 
\exp(-U(x))
\]
where $U(x) = \frac{v + d}{2} \log( 1 + \norm{x}^2 / v )$. 
The adapted unadjusted Langevin process $(\Gamma_t, X_t)_{t \ge 0}$ on $(0, 1) \times \R^d$ is defined by
\[
X_{t + 1}
= X_t - \Gamma_{t+1} \nabla U(X_t) + \sqrt{ 2 \Gamma_{t+1} } Z_{t + 1}
\]
where $\Gamma_{t+1} \in (0, 1)$ and $Z_{t + 1}$ is an independent standard normal random vector.
Subgeometric drift conditions have been shown for unadjusted Langevin in the non-adaptive case for heavy tailed target measures \citep{kamatani:2009}.

Let $\alpha > 0$ and $W(x) = (1 + \norm{x}^2)^{(v + d)/2}$.
By Ito's formula, for large enough $\norm{x}$, there is a constant $\e > 0$ such that the second term is bounded using the moment generating function of non-central chi-square random variables by
\eq{
&\E\left[ W^\alpha(X_{t + 1}) | \Gamma_{t + 1} = \gamma, X_t = x \right] - W^\alpha(x)
\\
&= \E\left[ 
\int_{0}^\gamma \nabla W(x)^\alpha \cdot dX_t
\right]
+ \E\left[ 
\int_{0}^\gamma \tr\left( \nabla^2 W(X_t)^\alpha X_t \right) dt
\right]
\\
&\le 
 \alpha ( v + d )  \left[ - \gamma_* (v + 2) + \alpha ( v + d ) + \e \right]
\left( 1 + \norm{x}^2 \right)^{\alpha(v + d)/2 - 1}.
}
It follows that for some constant $C_\alpha > 0$ and for all $x, \gamma$,
\eq{
\E\left[ W^\alpha(X_{t + 1}) | \Gamma_{t + 1} = \gamma, X_t = x \right] - W^\alpha(x)
\le 
C_{\alpha}
W^\alpha(x)^{1 - \frac{2}{\alpha (v + d)}}.
}
So then here
$
H^{-1}_{W(0)^{\alpha}, \phi}(t) = \left( \frac{2}{\alpha (v + d)} C_\alpha t + 1 \right)^{\alpha \frac{v + d}{2}}.
$
For $r \ge 1$, one has the lower bound for some constant $C > 0$
\[
\pi(  W \ge r )
\ge \frac{C}{r^{v/(v + d)}}.
\]
By Theorem~\ref{theorem:weak_lb_with_target_tails}, there are constants $M > 0$ such that with $\alpha$ chosen large enough
\[
\inf_{\Q \in \S(\X, \Y) } \W_{\norm{\cdot} \wedge 1}( \A_{\Q}^{(t)}(\gamma_0, 0, \cdot), \pi)
\ge \frac{M}{ \left( \frac{2}{\alpha (v + d)} C_\alpha t + 1 \right)^{ \frac{v}{\alpha - v/(v + d)} \frac{\alpha}{2}} }.
\]
Of particular interest is that the rate cannot be geometric even when considering weak convergence.
\end{example}

In certain situations, the tail probability decay on $\pi$ in \eqref{eq:target_tail_lb} may be difficult to establish. 
In this case, we consider finding a function that is not integrable with respect to $\pi$, but this results in a trade-off of only having a lower bound for a subsequence.
An analogous result will also hold in total variation.

\begin{theorem}
\label{theorem:weak_lb}
Assume for some function $W : \X \to [1, \infty)$ such that $\log(W)$ is Lipschitz and $\int_\X W d\pi = \infty$ but also for some $\alpha > 1$ and some concave function $\phi : (0, \infty) \to (0, \infty)$,
\begin{align}
(\P_\gamma W^{\alpha})(x) - W(x)^\alpha  \le \phi( W(x)^\alpha )
\end{align}
holds for all $(x, \gamma) \in \X \times \Y$. 
Then for $\beta \in (0, 1)$ with $\alpha > 1 + \beta$, there is a constant $M_* > 0$ and a subsequence $t_n \in \Z_+$ increasing to infinity such that,
\[
\inf_{\Q \in \S(\X, \Y) } 
\W_{d \wedge 1}\left(\A_{\Q}^{(t_n)}(\gamma_0, x_0, \cdot), \pi \right)
\ge \frac{M_*}{ \left( H_{W^\alpha(x_0), \phi}^{-1}(t_n) \right)^{ \frac{1 + \beta}{\alpha - 1 - \beta} } }.
\]
\end{theorem}

\begin{proof}
The inverse function theorem implies the derivative 
$
\frac{d}{ds} H_{ W(x_0)^\alpha, \phi}^{-1}(s) 
= \phi(H^{-1}_{ W(x_0)^\alpha, \phi}(s)).
$
Since $\phi$ has a linear upper bound,
$\frac{d}{dt} \log(H_{ W(x_0)^\alpha, \phi}^{-1}(t) ) \le m_{\alpha, x_0} < \infty$ for some $m_{\alpha, x_0} > 0$.
We then have
\[
H_{ W(x_0)^\alpha, \phi}^{-1}(t + 1)
\le \exp( m_{\alpha, x_0} ) H_{ W(x_0)^\alpha, \phi}^{-1}(t).
\]
Since $\int_\X W d\pi = \infty$, we can choose a constant $C$ so that
$C \exp( m_{\alpha, x_0} )^{-\frac{1 + \beta}{\alpha - 1 - \beta}} - (1 - \e)^{-\alpha} > 0$ and there is a sequence $(r_n)_n$ with $\lim_n r_n = \infty$ such that with $T_{n} = \{x : W(x) \ge r_n \}$,
\[
\pi(T_n)
\ge \frac{C}{r_n^{1 + \beta}}.
\]
For some $\e, \delta_\e \in (0, 1)$ as in Theorem~\ref{theorem:weak_lb_with_target_tails}, we have by Markov's inequality
\begin{align*}
\frac{1}{\delta_\e} \W_{d \wedge 1}\left(\A_{\Q}^{(t_n)}(\gamma_0, x_0, \cdot), \pi \right)
&\ge \inf_{\xi \in \C\left[ \A_{\Q}^{(t)}(\gamma_0, x_0, \cdot), \pi \right] } \xi( \{ x, y : d(x, y) > \delta_\e \})
\\
&\ge \frac{C}{r_n^{1 + \beta}} - \frac{H_{W(x_0)^\alpha, \phi}^{-1}(t) }{ (1 - \e)^{\alpha} r_n^{\alpha} }.
\end{align*}
holds for $t \in \Z_+$.
We can construct a subsequence $t_n \in \Z_+$ using $r_n$.
Due to monotonicity, we can choose $t_n \in \Z_+$ so that
\begin{align*}
H_{ W(x_0)^\alpha, \phi}^{-1}(t_n)^{\frac{1}{\alpha - 1 - \beta}}
\le r_n
&\le H_{ W(x_0)^\alpha, \phi}^{-1}(t_n + 1)^{\frac{1}{\alpha - 1 - \beta}}
\\
&\le \exp( m_{\alpha, x_0} )^{\frac{1}{\alpha - 1 - \beta}} H_{ W(x_0)^\alpha, \phi}^{-1}(t_n)^{\frac{1}{\alpha - 1 - \beta}}.
\end{align*}
Plugging in the bounds to $r_n$ with the chosen $C$ ensuring the lower bound is positive, the conclusion follows.
\end{proof}

\section{Subgeometric upper bounds for adaptive MCMC}
\label{section:upper_bounds}

This section is dedicated to studying conditions such that an upper bound convergence rate can be obtained for adaptive MCMC comparable to the lower bounds in the previous section.
We first consider a quantitative alternative to the widely-used diminishing adaptation condition \citep{roberts:rosenthal:2007} that is stronger in the sense that it requires a specified rate of decay.

\begin{definition}
An adaptive process satisfies \textit{quantitative diminishing adaptation} with a function $G : \Z_+ \to (0, \infty)$ strictly decreasing to 0 if for any $\e \in (0, 1)$, there is a nonnegative integer $s_\e$ such that with probability at least $1 - \e$,
\begin{equation}
\norm{\P_{\Gamma_{t + 1}}(x, \cdot) - \P_{\Gamma_t}(x, \cdot)}_{\text{TV}}
\le G(t)
\label{eq:da}
\end{equation}
holds for all $x \in \X$ and $t \ge s_{\e}$ and this event is Borel measurable.
\end{definition}

\noindent 
Since $\X$ is separable, a sufficient condition for Borel measurability here is lower semicontinuity of the total variation as a function in $x$.
Next, we consider a simultaneous version of a subgeometric drift condition on the Markov family.

\begin{definition}
A Markov family $(\P_\gamma)_{\gamma \in \Y}$ satisfies a \textit{simultaneous subgeometric drift condition} if there is a Borel function $V : \X \to [1, \infty)$ and a concave function $\phi : [0, \infty) \to [0, \infty)$ strictly increasing to infinity with $\lim_{v \to \infty} \phi(v)/v = 0$ and a constant $K \ge 0$ such that
\begin{align}
(\P_\gamma V)(x) - V(x)
\le -\phi( V(x) ) + K
\label{eq:sim_subgeo_drift}
\end{align}
holds for every $(x, \gamma) \in  \X \times \Y$.
\end{definition}

Here we assume $\lim_{v \to \infty} \phi(v)/v = 0$ to exclude the geometric case. 
Subgeometric drift conditions for Markov chains has been studied previously \citep{douc:etal:2004,jarner:roberts:2002} but we adjust the previous conditions to hold over feasible tuning parameters $\Y$. We now combine this drift condition with a simultaneous local contracting condition.

\begin{definition}
A Markov family $(\P_\gamma)_{\gamma \in \Y}$ satisfies a \textit{simultaneously locally contracting condition} on a set $S \subseteq \X \times \X$ if there is a constant $\alpha \in (0, 1)$ where
\begin{align}
\norm{ \P_\gamma(x, \cdot) - \P_\gamma(y, \cdot) }_{\text{TV}}
\le 1 - \alpha
\label{eq:local_contractive}
\end{align}
holds for all $(x, y) \in S$ and $\gamma \in \Y$.
\end{definition}

Local coupling conditions have been studied in the subgeometric case for Markov chains \citep{durmus:etal:2016}.
For example, a minorization condition can be used to verify the Markov family is simultaneously locally contracting (see \citep{roberts:rosenthal:2007}).
If the inverse is not well-defined, we will denote the generalized inverse of a real-valued function $F$ by $F^{-1}(\cdot) = \inf\{ x : F(x) \ge \cdot \}$.
Under these three conditions, we can establish an upper bound for the adaptation process. 

\begin{theorem}
\label{theorem:adapt_ub}
Assume the following assumptions hold for the Markov family $(\P_\gamma)_{\gamma \in \Y}$:
\begin{enumerate}
\item $\pi \P_\gamma = \pi$ for all $\gamma \in \Y$.

\item A simultaneously subgeometric drift condition \eqref{eq:sim_subgeo_drift} holds with a Borel function $V(\cdot)$ such that $\int_{\X} V d\pi < \infty$.

\item A simultaneous locally contracting condition \eqref{eq:local_contractive} holds on the set $S = \{ x, y \in \X \times \X : V(x) + V(y) \le R \}$ for $R = \phi^{-1}[ 2K/(1 - \delta) ]$ and some $\delta \in (0, 1)$.
\end{enumerate}
Additionally, assume the quantitative diminishing adaptation condition \eqref{eq:da} holds.
Then for any $\e \in (0, 1)$ and all $t$ sufficiently large,
\[
\norm{ \A_{\Q}^{(T_{\e, t} + t)}((\gamma_0, x_0), \cdot) - \pi }_{\text{TV}}
\le \frac{ 2 + \left[ r(1)/ r(0) \right] \left( V(x_0) + \int_{\X} V d\pi + K T_{\e, t} \right) + C}{ H_{1, \delta \phi}^{-1}\left( \frac{t}{ -\log( H_{1, \delta \phi}^{-1}(t) )/\log(1-\alpha) + 1} \right) } + 2 \e
\]
where $T_{\e, t} = s_{\e} + \lceil (1/G)^{-1}( t^2 / \e ) \rceil$ and
\eq{
&r(\cdot) = \delta \phi(H_{1, \delta \phi}^{-1}(\cdot)),
&C = \frac{r(1)}{r(0)}  \left\{ R + \frac{2 K r(1)}{r(0)} \right\}.
}
\end{theorem}

There are many ways to verify quantitative diminishing adaptation \eqref{eq:da} through concentration inequalities and almost sure convergence.
One specific way to verify quantitative diminishing adaptation \eqref{eq:da} when $\Y$ is a Euclidean space is through Lipschitz continuity properties of $\gamma \mapsto \P_\gamma$ combined with choosing an adaptation strategy so that $( \Gamma_{t + 1} - \Gamma_t ) / G(t)$ convergences almost surely. 
For example, if for each $x \in \X$, $\gamma \mapsto \P_{\gamma}(x, \cdot)$ is Lipschitz with constant $\beta_x$, then
\begin{align*}
\norm{\P_{\Gamma_{t + 1}}(x, \cdot) - \P_{\Gamma_t}(x, \cdot)}_{\text{TV}}
\le \left( \sup_{x \in \X} \beta_x \right) \norm{ \Gamma_{t + 1} - \Gamma_t}
\end{align*}
holds for all $x \in \X$ and \eqref{eq:da} follows immediately if $\sup_{x \in \X} \beta_x < \infty$.
The Lipschitz continuity used above can be verified under certain conditions when adapting the covariance in Gaussian random-walk proposals for Metropolis-Hastings \citep{andrieu:moulines:2006}.
A concrete example to satisfy the almost sure convergence of $( \Gamma_{t + 1} - \Gamma_t ) / G(t)$ is the widely-used adaptation with stochastic approximation when $\Gamma_{t + 1} = \Gamma_t + h_{t} L_t(\Gamma_0, X_0, \ldots, \Gamma_{t}, X_t)$ such that $L_t$ is bounded and $h_{t}$ is a step size tending towards zero with a specified rate.

Theorem~\ref{theorem:adapt_ub} requires satisfying quantitative diminishing adaptation \eqref{eq:da} with a sufficiently fast rate. 
Table~\ref{table:adapt_ub} compares approximate upper bounds for different combinations of $\phi(\cdot)$ and $G(\cdot)$.
%For example, if \eqref{eq:da} diminishes geometrically fast, then Theorem~\ref{theorem:adapt_ub} is capable of providing a comparable upper bound to the lower bounds in Section~\ref{section:lower_bounds}.
The upper and lower bounds may be also combined and in particular, Theorem~\ref{theorem:adapt_ub} can guarantee the adaptive process approximately achieves the lower bound rate if the adaptation diminishes sufficiently fast.
For example, if in addition to the assumptions of Theorem~\ref{theorem:adapt_ub}, there are constants $C, \kappa > 0$ such that 
\begin{align*}
&\pi( V \ge r ) \ge C r^{-\kappa},
&(\P_\gamma V^{2\kappa})(x) - V(x)^{2\kappa}  \le \phi( V(x)^{2\kappa} )
\end{align*}
holds for every $(x, \gamma) \in  \X \times \Y$. Then Theorem~\ref{theorem:lb_with_target_tails} and Theorem~\ref{theorem:adapt_ub} imply some constants $M^*, M_*, \alpha > 0$ such that
\begin{align*}
\frac{M_*}{H_{V(x_0)^{2\kappa}, \phi}^{-1}(T_{\e, t} + t) }
\le \norm{ \A_{\Q}^{( T_{\e, t} + t )}((\gamma_0, x_0), \cdot) - \pi }_{\text{TV}}
\le \frac{M^* T_{\e, t}}{ H_{1, \delta \phi}^{-1}\left( \frac{t}{ -\log( H_{1, \delta \phi}^{-1}(t) )/\log(1-\alpha) + 1} \right) } + 2 \e
\end{align*}
holds for large enough $t$ and $\e \in (0, 1)$. 
More interestingly, Theorem~\ref{theorem:weak_lb_with_target_tails} implies this also holds with alternative constants when $\X = \R^d$ with the total variation replaced by the Wasserstein distance corresponding to weak convergence.

%As an example, consider a target measure on $\R^d$ with potential $U : \R^d \to \R$ defined by $\pi(dx) \propto \exp(-U(x)) dx$ and Lyapunov function defined by $V(x) \equiv \exp(\kappa U(x))$ for $\kappa > 0$.
%Then with $\kappa < 1$, this can be used to obtain an upper bound and with $\kappa > 1$, this can be used to obtain a weak lower bound.

\begin{table}
\begin{center}
\renewcommand{\arraystretch}{2}
\begin{tabular}[t]{| c | c |}
\hline
\multicolumn{2}{|c|}{Example of upper bound rate from Theorem~\ref{theorem:adapt_ub}} 
\\
\hline
G(t) 
& $\phi(w) = c w^{\beta}, \beta \in (0, 1)$ 
\\ 
\hline
$\exp(-\alpha t), \alpha > 0$ 
& $\propto \frac{\log(t)}{( 1 + (1 - \beta) c t / \log(t))^{1/(1 - \beta)} }$ 
\\
\hline
$\exp(-t^\alpha), \alpha \in (0, 1)$ 
& $\propto \frac{ \log(t)^{1/\alpha} }{ ( 1 + (1 - \beta) c t / \log(t))^{1/(1 - \beta)} }$ 
\\
\hline
$t^{-\alpha}, \alpha > 1$ 
& $\propto \frac{ t^{2/\alpha} }{ ( 1 + (1 - \beta) c t / \log(t))^{1/(1 - \beta)} }$ 
\\
\hline
\end{tabular}
\end{center}
\caption{Upper bound convergence rate comparisons from Theorem~\ref{theorem:adapt_ub} for different combinations of $\phi(\cdot)$ and $G(\cdot)$. The table entries specify a convergence rate upper bound up to an explicit constant.}
\label{table:adapt_ub}
\end{table}

\begin{proof}[Proof of Theorem~\ref{theorem:adapt_ub}]
Let $(\Gamma_t, X_t)_{t \ge 0}$ be an adaptive process initialized at $x_0, \gamma_0$.
We first specify a finite adaptation plan $\Q^{T}$ with a time $T \in \Z_+$ defining a stopping point of adaptation.
This defines an adaptive process $(\Gamma'_t, X'_t)_{t \ge 0}$ initialized at $x_0, \gamma_0$ where for all $t \ge T$, $\Gamma'_t = \Gamma'_{T}$ and $(\Gamma'_t, X'_t) | (\Gamma'_s, X'_s)_{s \le t-1} \sim \delta_{\Gamma'_{T}}(\cdot) \P_{\Gamma'_{T}}(X'_{t-1}, \cdot)$ and $\delta_{\Gamma'_{T}}$ denotes the Dirac measure at $\Gamma'_{T}$.
Using this finite adaptation process, we have an upper bound via the triangle inequality
\begin{align}
&\norm{ \A_{\Q}^{(T + t)}((\gamma_0, x_0), \cdot) - \pi }_{\text{TV}} \nonumber
\\
&\le \norm{ \A_{\Q}^{(T + t)}((\gamma_0, x_0), \cdot) - \A_{\Q^T}^{(T + t)}((\gamma_0, x_0), \cdot) }_{\text{TV}}
+ \norm{ \A_{\Q^{T}}^{(T + t)}((\gamma_0, x_0), \cdot) - \pi }_{\text{TV}}.
\label{eq:triangle_ub}
\end{align}

We will bound each term on the right hand side of \eqref{eq:triangle_ub} separately.
For the first term in \eqref{eq:triangle_ub}, fix $\e \in (0, 1)$ and choose $T = T_{\e, t}$.

Let 
\[
B_{\e} = \{ \norm{\P_{\Gamma_{t + 1}}(x, \cdot) - \P_{\Gamma_t}(x, \cdot)}_{\text{TV}} \le G(t), x \in \X, t \ge s_{\e} \}
\] 
be the set where \eqref{eq:da} holds and is Borel measurable by assumption.
Using the triangle inequality, we have that
\begin{align*}
&\sup_{x \in \X}\norm{\P_{\Gamma_{T + t}}(x, \cdot) - \P_{\Gamma_{T}}(x, \cdot)}_{\text{TV}}
\\
&\le \sum_{s = 1}^{t} \sup_{x \in \X}\norm{\P_{\Gamma_{T + s}}(x, \cdot) - \P_{\Gamma_{T + s - 1}}(x, \cdot)}_{\text{TV}}
\\
&\le t G(T)
%\\
%&\le \frac{\e}{2t}.
\end{align*}
holds on $B_{\e}$.
Let $D = \{ u, v \in \X : u = v \}$ denote the diagonal set.
Proposition~\ref{prop:measurable_selection} implies there exists a family of Borel measurable joint kernels $(\K_{\gamma \gamma'} )_{\gamma, \gamma' \in \Y}$ on $\X \times \X$ such that for $1 \le s \le t$
\[
X_{T + s}, X'_{T + s} \mid \Gamma_{T + s} = \gamma, \Gamma_{T} = \gamma', X_{T + s - 1} = x, X'_{T + s - 1} = x' \sim \K_{\gamma, \gamma'}((x, x'), \cdot)
\]
and on the set $B_{\e}$,
\[
\K_{\Gamma_{T + s}, \Gamma_{T}}((x, x), D)
\ge 1 - \frac{\e}{t}
\]
and the diminishing adaptation assumption assumes measurability holds for all $x \in \X$.
Using the synchronous coupling for $\Gamma_t$ and $\Gamma'_t$ and since both processes are initialized at the same point, we can construct a coupling where $X_{s} = X_{s}'$ for $s \le T$ and using these joint kernels after $T$.
We then have the lower bound
\begin{align*}
\E\left( \prod_{k = 1}^{T + t} I_D(X_{k}, X'_{k}) I_{B_{\e}} \right)
%&= \E\left( \prod_{k = 1}^{T + t - 1} I_D(X_{k}, X'_{k}) \Prob\left( X_{t + T} = X'_{t + T} \mid X_{T + t - 1}, X'_{T + t - 1}, \Gamma_{T + t}, \Gamma_{T} \right) \mid S \right) 
%\\
&\ge \E\left( \prod_{k = 1}^{T + t - 1} I_D(X_{k}, X'_{k}) 
\inf_{x \in \X} \int_{y, y' \in D} \K_{\Gamma_{T + t}, \Gamma_{T}}((x, x), (dy, dy')) I_{B_{\e}} \right) 
\\
&\ge \left( 1 - \frac{\e}{t} \right) \E\left( \prod_{k = 1}^{T + t - 1} I_D(X_{k}, X'_{k}) I_{B_{\e}} \right)
\\
&\ge \left( 1 - \frac{\e}{t} \right)^t \Prob(B_{\e})
\\
&\ge \left( 1 - \e \right) \Prob(B_{\e})
\\
&\ge \left( 1 - \e \right)^2.
\end{align*}
Then it follows that the optimal coupling is controlled so that
\begin{align*}
\norm{ \A_{\Q}^{(T_{\e, t} + t)}((\gamma_0, x_0), \cdot) - \A_{\Q^{T_{\e, t}}}^{(T_{\e, t} + t)}((\gamma_0, x_0), \cdot) }_{\text{TV}}
&\le 
\Prob\left( X_{T + t} \not= X'_{T + t} \right)
\\
&\le 2 \e.
\end{align*}

To bound the second term in \eqref{eq:triangle_ub}, the following is adapted from previous arguments for subgeometric upper bounds for non-adapted Markov chains \citep{durmus:etal:2016}, but modified for adaptive MCMC.
The constants here are made explicit, but the general technique is not new.
There is a Borel measurable conditional total variation distance by Proposition~\ref{prop:measurable_selection} so that
\eq{
\norm{ \A_{\Q^{T}}^{(T + t)}((\gamma_0, x_0), \cdot) - \pi }_{\text{TV}}
&\le \E \left[ \norm{ \P_{\Gamma_{T}}^{t}(X_{T}, \cdot) - \pi }_{\text{TV}} \right].
}
Since $\phi$ is concave, it is subadditive so $\phi(V(x) + V(y)) \le \phi( V(x) ) + \phi( V(y) )$.
Since $\phi$ is strictly increasing, by the drift condition,
\eq{
(\P_\gamma V)(x) + (\P_\gamma V)(y) - [ V(x) + V(y) ]
&\le -[ \phi( V(x) ) + \phi( V(y) ) ] + 2K
\\
&\le -[ \phi( V(x) + V(y) ) ] + 2 K
}
holds for all $x, y \in \X$.
Using Lemma~\ref{lemma:simult_subgeo_drift},
\eq{
(\P_\gamma V)(x) + (\P_\gamma V)(y) - [ V(x) + V(y) ]
\le -\delta [ \phi( V(x) + V(y) ) ] + 2 K I_{S}(x, y).
}

Let $\tau_S = \inf\{ n \ge 1 : (X_n, Y_n) \in S \}$ be the first hit time to the set $S$.
For $n \in \Z_+$, let $\theta_n$ denote the shift operator applied $n$ times so that $\theta_n(X_{i}) = X_{i + n}$ for all $i \in \Z_+$.
Define the successive hit times to $S$ by 
\eq{
&\sigma_1 = \inf\{ n \ge 0 : (X_n, Y_n) \in S \},
&\sigma_2 = \tau_S \circ \theta_{\sigma_1},
\\
&\sigma_{n + 1} = \tau_S \circ \theta_{\sigma_1 + \cdots \sigma_n},
&\tau_n = \sum_{k = 1}^{n} \sigma_k
}
for each $n \in \Z_+$.
The inverse function theorem implies the derivative 
$
r(s)
= \frac{d}{ds} H^{-1}_{1, \delta \phi}(s) 
= \delta \phi(H^{-1}_{1, \delta \phi}(s))
$
for $s \ge 0$.
Thus, $H_{1, \delta \phi}^{-1}$ is convex since its derivative is monotone increasing by Lemma~\ref{lemma:H_properties}.
By Markov's inequality and Jensen's inequality,
\eq{
\Prob(\tau_m \ge t)
&\le \frac{ \E\left[ H_{1, \delta \phi}^{-1}\left( \frac{ 1 }{m} \sum_{k = 1}^m \sigma_k \right) \right]
}{
H_{1, \delta \phi}^{-1}(t/m)
}
\\
&\le \frac{ \frac{ 1 }{m} \sum_{k = 1}^m \E\left[ H_{1, \delta \phi}^{-1}( \sigma_k ) \right]
}{
H_{1, \delta \phi}^{-1}(t/m)
}.
}

For any $t, m \in \Z_+$ with $t \ge m$, the local contraction condition \eqref{eq:local_contractive} implies an upper bound via a coupling argument with \citep[Lemma 3.1]{jarner:tweedie:2001} so that for all $\gamma \in \Y$ and $x, y \in \X$,
\begin{align*}
\norm{ \P_{\gamma}^t(x, \cdot) - \P_{\gamma}^t(y, \cdot) }_{\text{TV}}
%&\le \inf_{\xi \in \C\left( \P_{\gamma}^t(x, \cdot), \P_{\gamma}^t(y, \cdot) \right)} \xi( \{ u, v : u \not= v, \tau_m < t \}) + \Prob(\tau_m \ge t)
%\\
&\le (1 - \alpha)^m + \Prob\left( \sum_{k = 1}^m \sigma_k \ge t \right)
\\
&\le (1 - \alpha)^m 
+ \frac{ \frac{ 1 }{m} \sum_{k = 1}^m \E\left[ H_{1, \delta \phi}^{-1}( \sigma_k ) \right]
}{
H_{1, \delta \phi}^{-1}(t/m)
}.
\end{align*}

By \citep[Proposition 2.2]{douc:etal:2004},
\[
\sup_{x, y \in S} \E_{x, y} \left( \sum_{i = 0}^{\tau_S - 1} r(i) \right)
\le \phi^{-1}( 2 K/(1 - \delta) ) + \frac{2 K r(1)}{r(0)}.
\]
Since $\phi$ is concave, we have that $h(\cdot) = \log( r(\cdot) )$ is concave and so centering with $g(\cdot) - g(0)$ is subadditive implying after exponentiating $r(s + t) \le r(s)r(t) / r(0)$ for all $t, s \ge 0$.
We then have the upper bound for $k \ge 2$,
\eq{
\E\left[ H_{1, \delta \phi}^{-1}( \sigma_k ) \right] - 1
&= \E\left[ \E_{X_{\tau_{k-1}}, Y_{\tau_{k-1}}} \left( \int^{\sigma_k}_0 r(s) ds \right) \right]
\\
&\le \E\left[ \E_{X_{\tau_{k-1}}, Y_{\tau_{k-1}}} \left( \sum_{i = 0}^{\tau_S - 1} r(i + 1) \right) \right]
\\
&\le \frac{r(1)}{r(0)} \E\left[ \E_{X_{\tau_{k-1}}, Y_{\tau_{k-1}}} \left( \sum_{i = 0}^{\tau_S - 1} r(i) \right) \right]
\\
&\le \frac{r(1)}{r(0)} \left\{ R + \frac{2 K r(1)}{r(0)} \right \}.
}
For $k = 1$, similarly, we have
\eq{
\E\left[ H_{1, \delta \phi}^{-1}( \sigma_1 ) \right] - 1
\le \frac{r(1)}{r(0)} \left\{ V(x) + V(y) + \frac{2 K r(1)}{r(0)} \right \}.
}
Combining these upper bounds,
\eq{
\E\left[ H_{1, \delta \phi}^{-1}\left( \frac{1}{m} \sum_{k = 1}^m \sigma_k \right) \right]
\le 1 + \frac{r(1)}{r(0)} \left\{  V(x) + V(y) + R + \frac{r(1)}{r(0)} ( 2 K ) \right\}.
}
The simultaneous subgeometric drift condition \eqref{eq:sim_subgeo_drift} implies
\[
\E\left[ V( X_{T} ) \right]
\le V(x_0) + K T.
\]
Choosing $m \equiv m_t = \lceil \log( H_{1, \delta \phi}^{-1}(t) ) / \log(1/(1-\alpha)) \rceil$, we have the upper bound
\eq{
&\norm{ \A_{\Q^{T}}^{(T + t)}((\gamma_0, x_0), \cdot) - \pi }_{\text{TV}}
\\
&\le (1 - \alpha)^{m_t} 
+ \frac{ 1 + \frac{r(1)}{r(0)} \left\{  V(x_0) + K T + \int V d\pi + R + \frac{r(1)}{r(0)} ( 2 K ) \right\} }{ H_{1, \delta \phi}^{-1}(t/m_t) }
\\
&\le \frac{ 2 + \frac{r(1)}{r(0)} \left\{  V(x_0) + K T + \int V d\pi \right\} + C  }{ H_{1, \delta \phi}^{-1}(t/m_t) }.
}
\end{proof}

%\begin{remark}
%Alternatives to quantitative diminishing adaptation \eqref{eq:da} can also be formulated such as an expected diminishing adaptation condition so that there is a function $G$ decreasing to $0$ such that for all $t \in \Z_+$,
%\[
%\sup_{x \in \X} \E\left( \norm{\P_{\Gamma_{t + 1}}(x, \cdot) - \P_{\Gamma_t}(x, \cdot)}_{\text{TV}} \mid X_t = x \right)
%\le G(t).
%\]
%In particular, the supremum is taken outside of the expectation and in a certain technical sense is weaker than \eqref{eq:da}.
%However, expected diminishing adaptation can potentially be more difficult to verify in practical examples due to requiring an upper bound on the conditional expectation.
%\end{remark}

\section{Example: adaptive Metropolis-Hastings independence sampler}
\label{section:toy_example}

In many cases, it is difficult to choose a proposal for Metropolis-Hastings that approximately matches the tail behavior of a complex target measure \citep{schmidler:woodard:2011} and adaptive MCMC is often employed.
The point of this toy example is to concretely demonstrate this scenario. 
We will use the upper and lower bounds on the convergence to investigate the sensitivity to different adaptation strategies.
%In particular, this example investigates where adaptation may fail to change the convergence behavior from a fixed tuning parameter choice.
Consider the target measure $\pi(dx) = \exp(-x) I_{[0, \infty)}(x) dx$.
Let $(\gamma_*, \gamma^*)$ be the interval for some potential tuning parameters $1 < \gamma_* <  \gamma^*$ and consider a Metropolis-Hastings Markov chain with independent proposal
$
\gamma \exp(-\gamma x) I_{[0, \infty)}(x)
$
and Markov kernel defined for $(x, \gamma) \in [0, \infty) \times (\gamma_*, \gamma^*)$ and all Borel sets $A \subseteq [0, \infty)$ by
\begin{align}
\P_\gamma(x, A)
= \int_A a_\gamma(x, y) \gamma \exp(-\gamma y) dy
+ \delta_{x}(A) R_\gamma(x)
\label{eq:mhi}
\end{align}
where the acceptance function is $a_\gamma(x, y) = \exp\left[ (\gamma - 1) (y - x) \right] \wedge 1$ and the rejection probability is $R_\gamma(x) = 1 - \int_0^\infty a_\gamma(x, y) \gamma \exp(-\gamma y) dy$.
Since we restrict $\gamma > 1$, the tail probabilities of the proposal and the Metropolis-Hastings kernel are lighter than the target.
Due to this restriction, we will have a polynomial lower bound over any possible adaptation plan.

\begin{proposition}
\label{proposition:lb_imh}
Let $\A_{\Q}^{(t)}(\gamma_0, x_0, \cdot)$ be the marginal of the adaptive independent Metropolis-Hastings process at time $t \in \Z_+$ from \eqref{eq:mhi} with adaptation parameter set $(\gamma_*, \gamma^*)$ and initialized at $x_0, \gamma_0 \in (0, \infty) \times (\gamma_*, \gamma^*)$.
Then for any $\e \in (0, 1)$, there is a $\delta_{\e} \in (0, 1)$
\begin{align*}
\inf_{\Q \in \S([0, \infty), (\gamma_*, \gamma^*)) } \inf_{\xi \in \C\left[ \A_{\Q}^{(t)}(\gamma_0, x_0, \cdot), \pi \right] } \xi( \{ x, y \in \X \times \X : |y - x| > \delta_{\e} \})
\ge (1 - \e) \frac{M_*}{ \left( c_* t + \exp(\alpha x_0) \right)^{\frac{1}{\alpha - 1}} }
\end{align*}
for $\alpha \in (1, \gamma_*)$ with $M_* = \alpha^{-1/(\alpha - 1)} - \alpha^{-\alpha/(\alpha - 1)}$ and $c_{*} = \gamma^* / (\alpha - 1)
+ \gamma_* / (\gamma_* - \alpha)
+ \gamma^* - 1$.
\end{proposition}

\begin{proof}
Define $W(x) = \exp(x)$, and we have by a standard computation $\pi(W(x) \ge r) = r^{-1}$.
Assume $1 < \gamma_* < \gamma$.
For $\alpha < \gamma$, the identity holds
\eq{
&(\P_\gamma W^{\alpha})(x) - W^{\alpha}(x)
\\
&= \int_{[0, \infty)}  \exp(\alpha y) \left\{ 1 \wedge \left[ \exp\left[ (\gamma - 1) ( y - x ) \right] \right] \right\} \gamma \exp(-\gamma y) dy 
- \exp(\alpha x) \E\left( a_\gamma(x, Y) \right).
}
We also have for any $\alpha < \gamma$,
\begin{align}
&\int_{[0, \infty)}  \exp(\alpha y) \left\{ 1 \wedge \left[ \exp\left[ (\gamma - 1) ( y - x ) \right] \right] \right\} \gamma \exp(-\gamma y) dy
\nonumber
\\
&= 
\frac{\gamma \exp\left[ (1 - \gamma) x \right]}{\alpha - 1} \int_0^x (\alpha - 1) \exp\left[ (\alpha - 1) y \right] dy
\nonumber
\\
&\hspace{1.5em}+ \frac{\gamma}{\gamma - \alpha} \int_x^\infty (\gamma - \alpha) \exp\left[ - ( \gamma - \alpha) y \right] dy
\nonumber
\\
&= 
\left\{ \frac{\gamma}{\alpha - 1} 
+ \frac{\gamma}{\gamma - \alpha}
\right\} 
\exp\left[ -(\gamma - \alpha) x \right] 
- \frac{\gamma}{\alpha - 1} \exp\left[ -(\gamma - 1) x \right].
\label{eq:mhi_identity}
\end{align}
Using \eqref{eq:mhi_identity} with $\alpha = 0$,
$\int a_\gamma(x, y) \gamma \exp(-\gamma y) dy
= \gamma \exp( (1-\gamma ) x)
+ (1 - \gamma) \exp(-\gamma x)$.
So then
\eq{
&(\P_\gamma W^{\alpha})(x) - W^{\alpha}(x)
\\
&= 
\left\{ \frac{\gamma}{\alpha - 1} 
+ \frac{\gamma}{\gamma - \alpha}
+ \gamma - 1
\right\} 
\exp\left[ -(\gamma - \alpha) x \right] 
- \frac{\gamma}{\alpha - 1} \exp\left[ -(\gamma - 1) x \right]
- \gamma W^{\alpha}(x)^{\frac{1 + \alpha - \gamma}{\alpha}}.
}
It follows that
\begin{align}
(\P_\gamma W^{\alpha})(x) - W^{\alpha}(x)
&\le \frac{\gamma^*}{\alpha - 1} 
+ \frac{\gamma_*}{\gamma_* - \alpha}
+ \gamma^* - 1.
\label{eq:mhi_drift}
\end{align}
With the upper bound in \eqref{eq:mhi_drift}, we can now use Theorem~\ref{theorem:lb_with_target_tails} with $\phi \equiv c_*$, $\kappa = 1$, and $\alpha \in (1, \gamma_*)$.
We then have the lower bound for $\e$ and some $\delta_\e$:
\eq{
\inf_{\Q \in \S([0, \infty), (\gamma_*, \gamma^*)) } \inf_{\xi \in \C\left[ \A_{\Q}^{(t)}(\gamma_0, x_0, \cdot), \pi \right] } \xi( \{ x, y \in \X \times \X : |y - x| > \delta_{\e} \})
\ge \frac{
M_* (1 - \e)
}{ 
\left[ c_* t + \exp(\alpha x_0) \right]^{\frac{1}{\alpha - 1}} 
}
}
holds for every $t \in \Z_+$ uniformly in the adaptation strategy $\Q$.
\end{proof}

In Table~\ref{table:mhi_lb}, we compute the lower bound in Proposition~\ref{proposition:lb_imh} for different choices of $(\gamma_*, \gamma^*)$.
The large values from Table~\ref{table:mhi_lb} illustrate that even in this toy example, it is possible to observe poor convergence behavior of adaptive MCMC with certain tuning parameter sets independently of the adaptation strategy.
However, this limitation on the convergence rate can be avoided if the adaptation plan is capable of crossing the critical boundary $\gamma = 1$.
By Theorem~\ref{theorem:weak_lb_with_target_tails}, the lower bound rate in Proposition~\ref{proposition:lb_imh} will be the same even when converging weakly.

\begin{table}
\begin{center}
\renewcommand{\arraystretch}{1.5}
\begin{tabular}[t]{| c | c | c | c |}
\hline
\multicolumn{4}{|c|}{Lower bound computations from Proposition~\ref{proposition:lb_imh}} \\
\hline
Iteration & $(\gamma_*, \gamma^*) = (3, 5)$ & $(\gamma_*, \gamma^*) = (4, 6)$ & $(\gamma_*, \gamma^*) = (8, 10)$ \\ 
\hline
$10^2$ & $0.0022$ & $0.014$ & $0.13$ 
\\
\hline
$10^3$ & $0.0007$ & $0.0063$ & $0.093$ 
\\
\hline
$10^4$ & $0.0002$ & $0.03$ & $0.067$ 
\\
\hline
$10^5$ & $0.00007$ & $0.0014$ & $0.048$ 
\\
\hline
\end{tabular}
\end{center}
\caption{
Lower bound computations from Proposition~\ref{proposition:lb_imh} for the adaptive Metropolis-Hastings independence sampler. Parameters used are defined in Proposition~\ref{proposition:lb_imh} with initialization $x_0 = 0$ and parameter values $\alpha = \gamma_* - .01$ and $\e = .001$.
}
\label{table:mhi_lb}
\end{table}

We now look at upper bounds from Section~\ref{section:upper_bounds} where we require adaptation is restricted to a compact set. This is a commonly used strategy in adaptive MCMC \citep{pompe:etal:2020}.

\begin{proposition}
\label{proposition:adaptive_mhi_ub}
For $t \in \Z_+$, let $\A_{\Q}^{(t)}(\gamma_0, x_0, \cdot)$ be the marginal of an adaptive independent Metropolis-Hastings process defined in Proposition~\ref{proposition:lb_imh} and assume for each $t \in \Z_+$, $\P_{\Gamma_{t + 1}}(x, \cdot) = \P_{\Gamma_{t}}(x, \cdot)$ for all $x \ge r$ for some $r > 0$.
Additionally, assume
\begin{align*}
\lim_{t \to \infty} \frac{ \left| \Gamma_{t + 1} - \Gamma_t \right| }{ G(t) }
= 0
\end{align*}
almost surely for some function $G(\cdot)$ strictly decreasing to $0$.
If $\gamma^* < 2 - \e$ for some $\e \in (0, 1)$, then for all $\delta \in (0, 1)$ and $t$ large enough,
\begin{align*}
\norm{ \A_{\Q}^{(T_{\delta, t} + t)}((\gamma_0, 0), \cdot) - \pi }_{\text{TV}}
&\le \frac{ M^* T_{\delta, t}}{ \left[ 
1 + c^* \frac{t}{\log( 1 + c^* t )/\log( (1-a^* )^{-1}) + 1}
\right]^{ \frac{1 - \e}{\gamma^* - 1}} }
+ \delta
\end{align*}
where $T_{\delta, t} = s_\delta + \lceil (1/G)^{-1}\left( J t^2 / \delta \right) \rceil$ for some $s_\delta \in \Z_+$ and some $c^*, J, M^*, a^*  > 0$.
\end{proposition}

\begin{proof}
We will use Theorem~\ref{theorem:adapt_ub} to establish the upper bound.
We first verify the quantitative diminishing adaptation condition \eqref{eq:da}.
Let $\phi : \R \to [0, 1]$ and for $x > 0$ and let $\psi_x(y) = \phi(y) - \phi(x)$. 
Then
\eq{
&\left| \P_{\gamma'}\phi(x) - \P_{\gamma}\phi(x) \right|
= \left| \P_{\gamma'}\psi_x(x) - \P_{\gamma}\psi_x(x) \right|
\\
&= \left| \int \psi_x(y) \left[ a_{\gamma'}(x, y) \gamma' \exp(-\gamma' y)
- a_{\gamma}(x, y) \gamma \exp(-\gamma y) \right] dy \right|
\\
&\le |\gamma' - \gamma| \int |y - x| \gamma' \exp(-\gamma' y) dy
+ \int |\gamma' \exp(-\gamma' y) - \gamma \exp(-\gamma y)| dy
\\
&\le (2/\gamma_*^2 + |x|) |\gamma' - \gamma|
+ \frac{1}{\gamma_*} |\gamma' - \gamma|.
}
Let $J = 2/\gamma_*^2 + r + \frac{1}{\gamma_*}$ and so quantitative diminishing adaptation \eqref{eq:da} since
\begin{align*}
\sup_{x \in \X} \norm{\P_{\Gamma_{t + 1}}(x, \cdot) - \P_{\Gamma_t}(x, \cdot)}_{\text{TV}}
&\le J |\Gamma_{t + 1} - \Gamma_{t}|.
%\\
%&\le \frac{J}{G(t)}.
\end{align*}

Next, we verify the simultaneous subgeometric drift condition. 
For $\e \in [0, 1)$, let $V^{1-\e}(x) = \exp( (1-\e) x)$, and using the identity \eqref{eq:mhi_identity}, for $\e \in (0, 1)$ and $2 - \e - \gamma^* > 0$,
\eq{
(\P_\gamma V^{1 - \e})(x) - V^{1 - \e}(x)
\le 
\left\{
\frac{\gamma^*}{\gamma_* - 1 + \e}
+ \gamma^* - 1
\right\}  
- \gamma_* V^{1-\e}(x)^{\frac{2 - \e - \gamma^*}{1 - \e}}.
}
Now we satisfy the simultaneous local contraction with a local minorization condition.
If $V^{1-\e}(x) \le R$ for $R > 1$, then $\exp((\gamma - 1) x) \le R^{\frac{\gamma - 1}{1-\e}}$, then
\begin{align*}
\inf_{ V^{1-\e}(x) \le R } \P_\gamma(x, \cdot)
&\ge \int_{\cdot} \left[ \frac{\exp((\gamma - 1)y)}{R^{\frac{\gamma^* - 1}{1-\e}}} \wedge 1 \right] \gamma \exp(-\gamma y) dy
\\
&\ge \gamma_* R^{-\frac{\gamma^* - 1}{1-\e}} \int_{\cdot} \left[ \exp((\gamma - 1)y) \wedge 1 \right] \exp(-\gamma y) dy
\\
&\ge \gamma_* R^{-\frac{\gamma^* - 1}{1-\e}}
\pi(\cdot).
\end{align*}
\end{proof}

If adaptation diminishes fast enough, Proposition~\ref{proposition:adaptive_mhi_ub} shows the upper bound rate is essentially $(1 + t)^{\frac{1-\e}{1-\gamma^*}}$ and depends on the largest tuning parameter $\gamma^*$. This is due to the adaptation plan possibly concentrating on $\gamma^*$, which is farthest from the optimal choice. On the other hand, the lower bound rate $(1 + t)^{\frac{1}{1-\gamma_*}}$ depends on the smallest tuning parameter $\gamma_*$ being closest to the optimal choice.
In particular, there can be a gap in the upper and lower bounds on the convergence characterized by potential tuning parameters.
In Table~\ref{table:mhi_ub}, we compare the upper bound convergence rates with $\e = .01$, $G(t) = \exp(-t)$, and $\delta = (1 + ct)^{\frac{1-\e}{1-\gamma^*}}$.
We observe that the upper bound is sensitive to the tuning of $\gamma^*$.

\begin{table}
\begin{center}
\renewcommand{\arraystretch}{1.4}
\begin{tabular}[t]{| c | c | c | c |}
\hline
\multicolumn{4}{|c|}{Upper bound computations from Proposition~\ref{proposition:adaptive_mhi_ub}} \\
\hline
 Iteration $t$ & $(\gamma_*, \gamma^*) = (1.2, 1.5)$ & $(\gamma_*, \gamma^*) = (1.2, 1.6)$ & $(\gamma_*, \gamma^*) = (1.2, 1.7)$ \\ 
 \hline
 $10^2$ & $4.54 \cdot 10^{-4}$ & $2 \cdot 10^{-2}$ & $2.52$ 
 \\
 \hline
 $10^3$ & $6.6 \cdot 10^{-5}$ & $4.9 \cdot 10^{-3}$ & $6.8 \cdot 10^{-1}$ 
 \\
 \hline
 $10^4$ & $3.2 \cdot 10^{-6}$ & $4 \cdot 10^{-4}$ & $7.4 \cdot 10^{-2}$ 
 \\
 \hline
 $10^5$ & $8.8 \cdot 10^{-8}$ &  $2 \cdot 10^{-5}$ & $5.4 \cdot 10^{-3}$ 
 \\
 \hline      
\end{tabular}
\end{center}
\caption{A comparison of the upper bounds from Proposition~\ref{proposition:adaptive_mhi_ub} for the adaptive Metropolis-Hastings independence sampler with $\e = .01$ and $G(t) = \exp(-t)$.}
\label{table:mhi_ub}
\end{table}

\section{Example: Adaptive random walk Metropolis}
\label{section:adaptive_rwm}

Adaptive random-walk Metropolis is a popular simulation algorithm for Bayesian statistics \citep{haario:2001}.
Let $U : \R^d \to \R$ and consider the target measure with normalizing constant $Z = \int_{\R^d} \exp(-U(x)) dx$ defined by $\pi(dx) = Z^{-1} \exp(-U(x)) dx$.
We make the following regularity assumptions on the target $\pi$ which have been used previously to show convergence results in MCMC \citep{douc:etal:2004}.
Let $\norm{\cdot}_F$ denote the Frobenius norm.

\begin{assumption}
\label{assumption:pi_regularity}
Suppose $U$ is twice continuously differentiable and there exists a minimum such that $x_*$ and $U(x_*) = 0$.
Let $m \in (0, 1)$ and assume there are constants $d_k, D_k, r > 0$ for $k = 1, 2, 3$ such that for all $x \in \R^d$ with $\norm{x} \ge R$ for some $R > 0$,
\begin{align}
&d_1 \norm{x}^m \le U(x) \le D_1 \norm{x}^m,
\label{assumption:pi_bound}
\\
&d_2 \norm{x}^{m-1} \le \norm{\nabla U(x)} \le D_2 \norm{x}^{m-1},
& \frac{\nabla U(x)}{ \norm{ \nabla U(x) } } \cdot \frac{x}{\norm{x}}
\ge r,
\label{assumption:grad_pi}
\\
&d_3 \norm{x}^{m-2} \le \norm{\nabla^2 U(x)}_F \le D_3 \norm{x}^{m-2}.
\label{assumption:hess_pi_bound}
\end{align}
\end{assumption}

While Assumption~\eqref{assumption:pi_regularity} is strong, $U(\cdot) \propto \norm{\cdot}^{m}$ and the Weibull distribution are examples (see \citep{fort:moulines:2000}).
Let $g$ be a standard Gaussian density.
For $\gamma \in \Y$, define the random-walk Metropolis Markov family for $x \in \R^d$ and Borel $A \subseteq \R^d$ by
\begin{align}
\P_\gamma(x, A)
= \int_{x + \gamma^{1/2} \xi \in A} a(x, x + \gamma^{1/2} \xi) g(\xi) d\xi
+ \delta_x(A) R_\gamma(x)
\label{eq:rwm}
\end{align}
with acceptance function $a(x, y) = \exp[ U(x) - U(y) ] \wedge 1$, and rejection probability $R_\gamma(x) = 1 - \int_{\R^d} a(x, x + \gamma^{1/2} \xi) g(\xi) d\xi$.
We define an adaptive random-walk Metropolis process by adapting the covariance of the proposal \citep{haario:2001} with dynamics
$\Gamma_t | (\Gamma_{s}, X_{s})_{s \le t-1}$ first updating the covariance matrix, and then $X_t | \Gamma_t, X_{t-1}$ updating the current state with random-walk Metropolis.

For the tuning parameter set $\Y$, we consider the set of symmetric positive definite matrices on $\R^d$ such that the eigenvalues are bounded by constants $\lambda_*, \lambda^* > 0$, that is,
\begin{align}
\Y
= \left\{ 
\gamma \in \R^{d \times d} : 
\lambda_* I
\le \gamma
\le \lambda^* I, \gamma^T = \gamma
\label{assumption:gamma_eigenvalues}
\right\}.
\end{align}
One example is to adapt a sample covariance matrix scaled by $h > 0$ using the following identity
\begin{align}
\Gamma_{t}
= \frac{h}{t} \sum_{s = 0}^{t} (X_s - \bar{X}_{t}) (X_s - \bar{X}_{t})^T
= \frac{t - 1}{t} \Gamma_{t-1} + \frac{h}{t + 1} (X_{t} - \bar{X}_{t-1}) (X_{t} - \bar{X}_{t-1})^T
\label{eq:sample_covariance}
\end{align}
where $\bar{X}_{t} = (t+1)^{-1} \sum_{s = 0}^{t} X_{s}$ \citep{haario:2001, andrieu:moulines:2006}.
%Moreover, the Cholesky factor can also be computed and updated in a computationally efficient way.
The set $\Y$ is convex and one way to ensure the updates remain in $\Y$ is to truncate the eigenvalues of \eqref{eq:sample_covariance}.

Under Assumption~\eqref{assumption:pi_regularity}, we first obtain a lower bound on the convergence rate for the adaptive random-walk Metropolis process.

\begin{proposition}
\label{proposition:adaptive_rwm_result}
For $t \in \Z_+$, let $\A_{\Q}^{(t)}(\gamma_0, x_0, \cdot)$ be the marginal of the adaptive random-walk Metropolis process initialized at $x_0, \gamma_0 \in \R^d \times \Y$ from the Metropolis-Hastings family \eqref{eq:rwm} and adaptation parameter set \eqref{assumption:gamma_eigenvalues}.
If Assumption~\eqref{assumption:pi_regularity} holds for $\pi$, then there are constants $c_*, M_* > 0  > 0$ such that
\eq{
\inf_{\Q \in \S(\X, \Y) } \W_{\norm{\cdot} \wedge 1}\left( \A_{\Q}^{(t)}(\gamma_0, x_0, \cdot), \pi \right)
\ge M_* \exp\left( -c_* {t}^{\frac{m}{2-m}}\right).
}
\end{proposition}

\noindent In order to proceed, we will first establish a simultaneous growth condition on the Markov family.

\begin{lemma}
\label{lemma:adaptive_rwm_lyap}
With $W(x) = \exp(U(x))$, there are constants $N, K > 0$ such that for any $\alpha > 1$ and $(x, \gamma) \in \X \times \Y$, there are constants $L^*_\alpha, M^* > 0$
\begin{align*}
(\P_\gamma W^\alpha)(x) - W^\alpha(x)
&\le M^* \alpha^{ 2/m - 1 } \left[ \alpha - N \right] \phi_K( W^\alpha(x) )
+ L^*_{\alpha}
\end{align*}
and if $\alpha < 1$ is small enough, there are constants $L_{*, \alpha}, M_* > 0$
\begin{align*}
(\P_\gamma W^\alpha)(x) - W^\alpha(x)
&\le M_* \alpha^{ 2/m - 1 } \left[ \alpha - N \right] \phi_K( W^\alpha(x) )
+ L_{*, \alpha}
\end{align*}
where $\phi_{K}$ is defined for all $w \ge 1$,
\[
\phi_{K}(w) = \frac{w + K}{ \log\left[ w + K \right]^{ 2/m - 2 } }.
\]
\end{lemma}

\begin{proof}
Let $R > 0$ and let $C_x = \{ \xi \in \R^d : \norm{\xi}^2 \le R \log( \norm{x} ) \}$.
We will control
\eq{
\frac{ (\P_\gamma W^\alpha)(x) }{W^\alpha(x)} - 1
&= \int_{\R^d} \left[ \frac{ W^\alpha(x +  \gamma^{1/2} \xi) }{ W^\alpha(x) } - 1 \right] a( x, x + \gamma^{1/2} \xi )g(\xi) d\xi
\\
&= \int_{C_x} \left[ \frac{ W^\alpha(x +  \gamma^{1/2} \xi) }{ W^\alpha(x) } - 1 \right] a( x, x + \gamma^{1/2} \xi )g(\xi) d\xi
\\
&+ \int_{C_x^c} \left[ \frac{ W^\alpha(x +  \gamma^{1/2} \xi) }{ W^\alpha(x) } - 1 \right] a( x, x + \gamma^{1/2} \xi )g(\xi) d\xi.
}
The first term is controlled through Taylor expansions.
For $\norm{x}$ large enough and $\xi \in C_x$,
$
\norm{x + t \gamma^{1/2} \xi}
\ge \norm{x}/2
$
using the assumptions on the eigenvalues of $\gamma$.
Similar to \citep[Lemma B.4]{fort:moulines:2000}, using the fundamental theorem of calculus and \eqref{assumption:grad_pi}
\eq{
&\sup_{\xi \in C_r} \frac{W^\alpha(x + \gamma^{1/2} \xi) }{W^\alpha(x)}
&\le 1 + \norm{\gamma} \sup_{\xi \in C_r, t \in [0, 1]} \norm{\xi} \sup_{\xi \in C_r} \frac{W^\alpha(x + t \gamma^{1/2} \xi) }{W^\alpha(x)} \int_0^1 \norm{x + t \gamma^{1/2} \xi}^{m - 1} dt.
}
It follows that
\eq{
\lim_{r \to \infty} \sup_{\norm{x} \ge r} \sup_{\xi \in C_r} \frac{W^\alpha(x + \gamma^{1/2} \xi) }{W^\alpha(x)}
&\le 1.
}
Using the fundamental theorem of calculus twice with \eqref{assumption:grad_pi} and \eqref{assumption:hess_pi_bound}, there is a constant $M > 0$ such that for large enough $\norm{x}$ and all $\xi \in C_r$
\eq{
&\left| \frac{ W^\alpha(x + \gamma^{1/2} \xi) }{W^\alpha(x)}
-  1 - \alpha \gamma^{1/2} \nabla U(x) \cdot \xi \right|
\\
&\le
\int_0^1 \frac{W^\alpha(x + t \gamma^{1/2} \xi) }{W^\alpha(x)} \{ \alpha^2 [ \gamma^{1/2} \nabla U(x + t \gamma^{1/2} \xi) \cdot \xi ]^2 
\\
&\hspace{.4cm}+ \alpha \xi \cdot \gamma^{1/2} \nabla^2 U(x + t \gamma^{1/2} \xi) \gamma^{1/2} \xi \} (1 - t) dt
\\
&\le
M \left( \alpha^2 \norm{x}^{2(m-1)} \norm{\xi}^2 \right).
}
Here we used again that for $\norm{x}$ large enough and $\xi \in C_x$,
$
\norm{x + t \gamma^{1/2} \xi}
\ge \norm{x}/2.
$
After integrating, we have shown that there is a constant $M > 0$ such that
\eq{
&\int_{C_x} \left[ \frac{ W^\alpha(x +  \gamma^{1/2} \xi) }{ W^\alpha(x) } - 1 \right] a( x, x + \gamma^{1/2} \xi )g(\xi) d\xi
\\
&\le \alpha \int_{C_x} \gamma^{1/2} \nabla U(x) \cdot \xi a(x, x + \gamma^{1/2} \xi) g(\xi) d\xi
+ M \alpha^2 \norm{x}^{2(m - 1)}.
}

The second term is controlled by the Gaussian decay of the proposal.
Assumption~\ref{assumption:grad_pi} implies $U$ is Lipschitz and for some $c > 0$
\[
\int_{C_x^c} \left[ \frac{ W^\alpha(x +  \gamma^{1/2} \xi) }{ W^\alpha(x) } - 1 \right] a( x, x + \gamma^{1/2} \xi )g(\xi) d\xi
\le \int_{C_x^c} \exp( c \norm{\xi} ) g(\xi) d\xi.
\]
For large enough $\norm{x}$ this implies for constants $c, C > 0$ that
\[
\int_{C_x^c} \left[ \frac{ W^\alpha(x +  \gamma^{1/2} \xi) }{ W^\alpha(x) } - 1 \right] a( x, x + \gamma^{1/2} \xi )g(\xi) d\xi
\le C \exp( - c \log(\norm{x}) R  ).
\]
We can tune $R$ depending on $\alpha, m$ so that for some constant $M > 0$
\[
\int_{C_x^c} \left[ \frac{ W^\alpha(x +  \gamma^{1/2} \xi) }{ W^\alpha(x) } - 1 \right] a( x, x + \gamma^{1/2} \xi )g(\xi) d\xi
\le \alpha^2 M \norm{x}^{2(m - 1)}.
\]

Let $r_\gamma(x) = \{ y \in \R^d : U(x) < U(x + \gamma^{1/2} y) \}$ denote the rejection region.
Combining these results, we have shown that
\eq{
&\int_{\R^d} \left[ \frac{ W^\alpha(x +  \gamma^{1/2} \xi) }{ W^\alpha(x) } - 1 \right] a( x, x + \gamma^{1/2} \xi )g(\xi) d\xi
\\
&\le \alpha \int_{\R^d} \gamma^{1/2} \nabla U(x) \cdot \xi a(x, x + \gamma^{1/2} \xi) g(\xi) d\xi
+ M \alpha^2 \norm{x}^{2(m - 1)}
\\
&\le \alpha \int_{r_\gamma(x) } \gamma^{1/2} \nabla U(x) \cdot \xi \left( \exp[ U(x) - U(x + \gamma^{1/2} \xi) ] - 1 \right) g(\xi) d\xi
+ M \alpha^2 \norm{x}^{2(m - 1)}.
}
Here we used that $\int_{\R^d} \xi g(\xi) d\xi = 0$.
Using a Taylor expansion $U(x + \gamma^{1/2} \xi) - U(x)$ to get for some $M > 0$
\eq{
&\int_{\R^d} \left[ \frac{ W^\alpha(x +  \gamma^{1/2} \xi) }{ W^\alpha(x) } - 1 \right] a( x, x + \gamma^{1/2} \xi )g(\xi) d\xi
\\
&\le -\frac{\alpha \norm{\nabla U(x)}^2}{2} \int_{r_\gamma(x) \cap C_x } \left( \gamma^{1/2} \frac{\nabla U(x)}{\norm{\nabla U(x)}} \cdot \xi \right)^2 g(\xi) d\xi
+ M \alpha^2 \norm{x}^{2(m - 1)}
\\
&\le -\frac{ \alpha d_2^2 \norm{x}^{2 (m - 1)}}{2} \int_{r_\gamma(x) \cap C_x } \left( \gamma^{1/2} \frac{\nabla U(x)}{\norm{\nabla U(x)}} \cdot \xi \right)^2 g(\xi) d\xi
+ M \alpha^2 \norm{x}^{2(m - 1)}.
}
The proof of \citep[Lemma B.3]{fort:moulines:2000} allows us to get a constant $b > 0$ such that for large enough $\norm{x}$
\[
\int_{r_\gamma(x) \cap C_x} \left( \frac{\nabla U(x)}{\norm{\nabla U(x)}} \cdot \xi \right)^2 g(\xi) d\xi
\ge b.
\]

Applying \eqref{assumption:pi_bound} and \eqref{assumption:grad_pi}, we have shown that there are constants $M_0, N_0 > 0$ such that $\norm{x}$ sufficiently large
\begin{align*}
(\P_\gamma W^\alpha)(x) - W^\alpha(x)
&\le M_0 \alpha \left[ \alpha -  N_0 \right] W^\alpha(x) \norm{x}^{2(m-1)}.
\end{align*}
By Assumption~\ref{assumption:pi_bound}, $\norm{x}^m$ is proportional to $U(x)$ and also
\begin{align*}
U(x) = \alpha^{-1} \log(W^\alpha(x)).
\end{align*}
Therefore, modifying the previously defined constants, if $\alpha$ is small enough, there are constants $M_0, N_0, K_m > 0$ such that $\norm{x}$ sufficiently large
\begin{align*}
(\P_\gamma W^\alpha)(x) - W^\alpha(x)
&\le M_0 \alpha^{ 2/m - 1 } \left[ \alpha - N_0 \right] \frac{W^\alpha(x)}{ \log(W^\alpha(x))^{ 2/m - 2 } }.
\end{align*}
With $\norm{x}$ large enough, for some $K > 0$, $W^\alpha(x) \le W^\alpha(x) + K \le 2 W^\alpha(x)$ and we obtain the result using $\phi_K$.
This completes the proof for large $\norm{x}$ and for small $\norm{x}$, we have by continuity, the sub-level sets of $W$ are compact and $(\P_\gamma W^\alpha)(x) - W^\alpha(x)$ is bounded on compact sets.
The case for large $\alpha$ is similar.
\end{proof}

\noindent We may now apply Lemma~\ref{lemma:adaptive_rwm_lyap} to obtain the lower bound.

\begin{proof}[Proof of Proposition~\ref{proposition:adaptive_rwm_result}]
We will complete the proof with $D_1 = d_1 = 1$ as the general results in changes to the constants.
Changing to polar coordinates, we have for $r$ large enough
\eq{
\pi\left( \exp(U(x)) \ge r \right)
&\ge \pi\left( \norm{x}^m \ge \log(r) / d_1 \right)
\\
&\ge \frac{2 \pi^{d/2}}{Z \Gamma(d/2)} \int_{s^m \ge \log(r)} s^{d - 1} \exp(-U(s)) ds
\\
&\ge \frac{2 \pi^{d/2}}{Z \Gamma(d/2)} \int_{s^m \ge \log(r) } s^{d - 1} \exp(-s^m) ds
\\
&\ge \frac{2 \pi^{d/2}}{Z m \Gamma(d/2)} \frac{1}{r}.
}
where $Z$ is the normalizing constant and $\Gamma(t) = \int_0^\infty u^{t - 1} \exp(-u) du$ for $t > 0$ is the Gamma function.

By Lemma~\ref{lemma:adaptive_rwm_lyap}, for $\alpha$ sufficiently large, we have constants $M, K_m > 0$ depending on $\alpha, m$ such that
\eq{
(\P_\gamma W^\alpha)(x) - W^\alpha(x)
&\le M \frac{W^\alpha(x) + K_m}{ \log(W^\alpha(x) + K_m)^{ 2/m - 2 } }
}
holds for all $(x, \gamma) \in \R^d \times \Y$.
Therefore, there are constants $c, M' > 0$ such that
\eq{
H_{W^\alpha(x_0), \phi}(u) 
&= \frac{1}{M} \int_{W^\alpha(x_0)}^{u} \frac{ \log(x + K_m)^{2/m - 2} }{ x + K_m} dx
\\
H_{W^\alpha(x_0), \phi}^{-1}(t) 
&\le M' \exp(\alpha U(x_0)) \exp\left( c t^{\frac{m}{2-m}}\right).
}
The lower bound then follows by Theorem~\ref{theorem:weak_lb_with_target_tails}.
\end{proof}

We investigate now an upper bound with the quantitative diminishing adaptation condition \eqref{eq:da} that can approximately achieve the lower bound rate.
The following upper and lower bounds show that the convergence of adaptive random-walk Metropolis in this situation is not geometric.
One drawback is that we do not obtain explicit constants in the upper and lower bounds.

\begin{proposition}
For $t \in \Z_+$, let $\A_{\Q}^{(t)}(\gamma_0, x_0, \cdot)$ be the marginal of an adaptive random-walk Metropolis process as in Proposition~\ref{proposition:adaptive_rwm_result}.
Assume
\[
\lim_{t \to \infty} \frac{ \norm{ \Gamma_{t + 1} - \Gamma_t }_F }{ G(t) }
= 0
\]
almost surely with $G(\cdot)$ strictly decreasing to $0$.
Then there are constants $M^*$ depending on $x_0$ and $c^*, J^*, d^* > 0$ such that for all $\e \in (0, 1)$ and all $t$ sufficiently large,
\[
\W_{\norm{\cdot} \wedge 1}\left( \A_{\Q}^{(T_{\e, t} + t)}((\gamma_0, x_0), \cdot), \pi \right)
\le M^* T_{\e, t}
 \exp\left[ -c^* t^{\frac{m}{2-m} \left( 1 - \frac{m}{2-m} \right)} \right]
 + \e
\]
where $T_{\e, t} = S_{\e} + \lceil (1/G)^{-1}\left( J^* t^2 / \e \right) \rceil$ for some $S_{\e} \in \Z_+$.
\end{proposition}

\begin{proof}
We will apply Theorem~\ref{theorem:adapt_ub} to obtain the conclusion.
Choosing $\alpha < 1$ sufficiently small in the drift condition from Lemma~\ref{lemma:adaptive_rwm_lyap}, gives the simultaneous drift condition ensures $\pi$-integrability of the drift function.
Using the compactness of the sublevel sets and continuity of the proposal density shows a simultaneous local minorization using Lebesgue measure holds.

It remains to verify quantitative diminishing adaptation \eqref{eq:da}.
For $\gamma \in \Y$, define 
\[
f_{\gamma}(y) = \frac{1}{(2\pi)^{d/2} \det(\gamma )^{1/2} } \exp\left( - \frac{1}{2} y^T {\gamma}^{-1} y \right).
\]
Following \citep[Lemma 13]{andrieu:moulines:2006}, the mean value theorem gives the upper bound
\begin{align*}
\int_{\R^d} \left| f_{\gamma'}(y) - f_{\gamma}(y)
\right| dy
&\le \frac{1}{2} \int_{\R^d} \int_0^1 f_{\gamma_t}(y) \left| \tr\left( \gamma_t^{-1} (\gamma' - \gamma)  + \gamma_t^{-1} y y^T \gamma_t^{-1} ( \gamma' - \gamma ) \right) \right| dt dy
\\
&\le \frac{d}{\lambda_*} \norm{\gamma' - \gamma}_F.
\end{align*}
Since the proposal is symmetric, then for Borel $\phi : \X \to [0, 1]$, let $\psi(x, y) = ( \phi(y) - \phi(x)) a(x, y)$ and
\eq{
\P_{\gamma'} \phi(x) - \P_{\gamma} \phi(x)
&= \int_{\R^d} \psi(x, y) f_{\gamma'}(y) dy - \int_{\R^d} \psi(x, y) f_{\gamma}(y)  dy
\\
&\le \int_{\R^d} \left| f_{\gamma'}(y) - f_{\gamma}(y)
\right| dy
\\
&\le \frac{d}{\lambda_*}  \norm{\gamma' - \gamma}_F.
}
Taking the supremum over $\phi$, we then have for each $t \in \Z_+$,
\eq{
\sup_{x \in \X} \norm{ \P_{\Gamma_{t + 1}}(x, \cdot) - \P_{\Gamma_t}(x, \cdot) }_{\text{TV}}
&\le \frac{d}{\lambda_*} \norm{ \Gamma_{t + 1} - \Gamma_t }_F.
%\\
%&\le \frac{J^*}{G(t)}.
}
\end{proof}

\section{Final discussion}
\label{section:conclusion}

The general weak lower bound convergence rates developed here in combination with total variation upper bounds for adaptive MCMC can provide useful guidance in designing adaptation strategies in MCMC.
We showed that the lower bound for weak convergence in Theorem~\ref{theorem:weak_lb_with_target_tails} can produce the same rate as the lower bound in total variation from Theorem~\ref{theorem:lb_with_target_tails}.
We also used a novel quantitative diminishing adaptation condition \eqref{eq:da} to show these lower bounds can be accompanied by upper bounds with subgeometric convergence rates.
Our contributions are useful not only in understanding the convergence of adaptive MCMC, but also for gaining intuition for constructing adaptation strategies in practice.

Choosing an optimal adaptation strategy for an adaptive MCMC simulation remains a difficult task in general and our subgeometric upper bounds are limited by requiring the adaptation to diminish sufficiently fast according to \eqref{eq:da}.
While this is to be expected, some interesting future research directions could include finding more precise classes of adaptation strategies that are capable of achieving upper bounds that can approximately match the lower bound rate.
Another area of interest is studying requirements on the adaptation that result in geometric convergence rates for adaptive MCMC.

\section*{Funding information}
This work was partially funded by NSERC Discovery Grant RGPIN-2019-04142.

\begin{appendix}

\section*{Supporting technical results}

The following is a technical result to ensure Borel measurability of conditional Wasserstein distances used in adaptive MCMC.

\begin{proposition}
\label{prop:measurable_selection}
Let $\X, \Y$ be a Polish spaces.
Assume $(x, \gamma) \mapsto K_{x, \gamma}$ is Borel measurable where $K_{x, \gamma}$ is a Borel probability measure on $\X$ and $x, \gamma \in \X \times \Y$.
Let $c : \X \times \X \to [0, \infty)$ be a lower semicontinuous function and for each $x, \gamma, x', \gamma' \in \X \times \Y$, let
\[
\W_{c}\left( K_{x, \gamma}, K_{x', \gamma'} \right)
= \inf_{\xi \in \C\left( K_{x, \gamma}, K_{x', \gamma'} \right) } \int_{\X \times \X} c(u, v) \xi(du, dv)
= \int_{\X \times \X} c(u, v) \xi^*_{x, \gamma, x', \gamma'}(du, dv).
\]
for an optimal coupling $\xi^*_{x, \gamma, x', \gamma'}$.
Then there is a Borel measurable choice of the the function $x, \gamma, x', \gamma' \mapsto \xi^*_{x, \gamma, x', \gamma'}$.
\end{proposition}

\begin{proof}
Let $T$ be the set of Borel optimal couplings $\xi^*$ on $\X \times \X$ satisfying
\[
\inf_{ \xi \in \C\left( \mu, \nu \right) } \int_{\X \times \X} c(u, v) \xi(du, dv)
= \int_{\X \times \X} c(u, v) \xi^*(du, dv)
\]
for some Borel probability measures $\mu, \nu$ on $\X$.
Since $\X$ is Polish, then $T$ is a Polish space.
Let $M_1(\X)$ be the set of probability measures on $\X$.
Let $\Phi : T \to M_1(\X) \times M_1(\X)$ be the function defined by $\Phi(\xi) = (\mu, \nu)$ for the optimal couplings $\xi \in \C\left( \mu, \nu \right)$.
Then the pre-images $\Phi^{-1}(\mu, \nu)$ are compact nonempty sets of optimal couplings \citep[Corollary 5.21]{villani:2008}.
By construction, $\Phi$ is continuous with respect to the weak topology as it is the projection onto the first and second coordinates.
Since $\Phi$ is continuous it is weakly Borel measurable so that $\{ \Phi(\xi) \cap U \not= \varnothing \}$ is Borel measurable for every open set $U$.
By Kuratowski and Ryll-Nardzewski \citep{kuratowski:1965}, there is a Borel measurable selection of the inverse $\Phi^{-1}$.
Now let $\psi : x, \gamma, x', \gamma' \mapsto ( K_{x, \gamma}, K_{x', \gamma'} )$ be Borel measurable and thus $\Phi^{-1}( \psi )$ is Borel measurable completing the proof.
\end{proof}

The following provides useful properties for the function $H_{1, \phi}$ and $\phi$ defined in Section~\ref{section:lower_bounds}.

\begin{lemma}
\label{lemma:H_properties}
Let $\phi : (0, \infty) \to (0, \infty)$ be concave and define for $w \ge 1$,
\begin{align}
H_{1, \phi}(w) = \int_{1}^w \frac{1}{\phi(v)} dv.
\end{align}
Then $\phi$ is non-decreasing and $H_{1, \phi}(\cdot)$ is strictly increasing to $\infty$.
\end{lemma}

\begin{proof}
The fundamental theorem of calculus implies $H_{1, \phi}(\cdot)$ is strictly increasing and then this implies that $H_{1, \phi}^{-1}(\cdot)$ exists.
Since $\phi$ is positive and concave, it is sublinear and so $H_{1, \phi}(\cdot)$ increases to $\infty$.
We need to show that $\phi$ is non-decreasing.
Suppose by contradiction that for some $u < v$ that $\phi(v) < \phi(u)$.
Let $\partial \phi$ denote a subgradient of $\phi$ and then $\phi(v) - \phi(u) = \partial \phi(\bar{u}) (v - u) < 0$ for some $\bar{u}$ in between $u, v$.
For any $v' \ge v$, $\phi(v') - \phi(v) \le \partial \phi(v) (v' - v) \le \partial \phi(\bar{u}) (v' - v)$ since $\phi$ is concave and its subgradients are not increasing.
For sufficiently large $v'$, we would then have $\phi(v') - \phi(v) \le - \phi(v) - 1$ which is a contradiction to $\phi \ge 0$.
\end{proof}

\noindent The next simple lemma is used for drift conditions.

\begin{lemma}
\label{lemma:simult_subgeo_drift}
Assume there is a Borel function $V : \X \to [1, \infty)$, an strictly increasing function $\phi : [0, \infty) \to [0, \infty)$, and a constant $K \in (0, \infty)$ such that 
\begin{align}
(\P_\gamma V)(x) - V(x)
\le -\phi( V(x) ) + K
\end{align}
holds for every $(x, \gamma) \in  \X \times \Y$.
Then for any $\delta \in (0, 1)$ and $C_\delta = \{x \in \X : V(x) \le \phi^{-1}( K/(1 - \delta) ) \}$,
\[
(\P_\gamma V)(x) - V(x)
\le -\delta \phi( V(x) ) + K I_{C_\delta}(x)
\]
holds for all $x \in \X$ and all $\gamma \in \Y$.
\end{lemma}

\begin{proof}
Let $x \in \X$ and if $\phi(V(x)) \ge R$ with $R =  K / (1 - \delta)$, then by the drift condition,
\eq{
(\P_\gamma V)(x) - V(x)
\le \left( \frac{K}{R} - 1 \right) \phi( V(x) )
\le -\delta \phi( V(x) ).
}
If $\phi(V(x)) \le R$, then the result follows from the assumed drift condition since $\delta < 1$.
\end{proof}

%\noindent The following is a standard result in topology.
%\begin{lemma}
%\label{lemma:distance_to_set}
%Let $A$ be a closed set in $\R^d$.
%Then for any $x \notin A$,
%$
%d(x, A)
%= d(x, \partial A).
%$
%\end{lemma}

\end{appendix}

\bibliography{references.bib}

\end{document}